\documentclass[11pt]{amsart} 
\usepackage{amsmath,amssymb,amsthm} 

\usepackage{thm-restate}

\usepackage{graphicx}
\usepackage{a4wide,comment}
\usepackage{amssymb}
\usepackage{xcolor,url}
\usepackage[all]{xy}
\usepackage[normalem]{ulem} 
\newtheorem{theorem}{Theorem}[section]    
\newtheorem{lemma}[theorem]{Lemma}          
\newtheorem{proposition}[theorem]{Proposition}  
\newtheorem{claim}[theorem]{Claim}  
\newtheorem{example-definition}[theorem]{Example-Definition} 

\newtheorem{corollary}[theorem]{Corollary} 

\theoremstyle{definition}
\newtheorem{definition}[theorem]{Definition}
\newtheorem{remark}[theorem]{Remark}  
\newtheorem{question}[theorem]{Question}

\newtheorem{assump}[theorem]{Assumption}  
\newtheorem*{remark-no-number}{Remark}             
\newtheorem*{definition*}{Definition}     
\newtheorem*{example*}{Example}     
\numberwithin{equation}{section}

\newcommand{\Int}{{\rm Int}}
\newcommand{\mF}{\mathcal{F}}

\newcommand{\Q}{\mathbb{Q}}
\newcommand{\Z}{\mathbb{Z}}

\newcommand{\wS}{\widetilde S}
\newcommand{\wP}{\widetilde P}
\newcommand{\wphi}{\widetilde \phi}

\newcommand{\vphi}{\varphi}
\newcommand{\wvphi}{\widetilde{\varphi_L}}
\newcommand{\wvphiL}{\widetilde{\varphi_L}}
\newcommand{\Diff}{{{\mathcal D}{\it iff}}^+}
\newcommand{\Homeo}{{{\mathcal H}{\it omeo}}^+}
\newcommand{\MCG}{{\mathcal{MCG}}}

\newcommand{\cL}{{\bf L}}

\newcommand{\id}{{\rm id}}
\newcommand{\f}{{\bf f}}

\title{On the fractional Dehn twist coefficients of branched coverings}
\author[Tetsuya Ito and Keiko Kawamuro]{Tetsuya Ito and Keiko Kawamuro}

\begin{document}

\begin{abstract}
We discuss how the fractional Dehn twist coefficient behaves under a fully ramified branched covering of an open book, and give applications to both topological and contact 3-manifolds. In particular, we show that non-right-veering closed braids represent virtually loose transverse links. 
\end{abstract}

\maketitle

\section{Introduction}

Let $S$ be an oriented compact surface with nonempty boundary and $P \subset S$ be a possibly empty finite set of points.
Let $\MCG(S \setminus P)$ be the mapping class group of the punctured surface $S\setminus P$. For a connected boundary component $C$ of $S$, the \emph{fractional Dehn twist coefficient} (FDTC, for short) with respect to $C$ is a map $$c(-,C):\MCG(S \setminus P) \rightarrow \Q$$ which quantifies twisting of a mapping class along $C$ (see \cite{HKM} for the definition). In the special case where $S=D_n$ an $n$-punctured disk, the braid group $B_n$ and $\MCG(D_n)$ are isomorphic. With this identification, we can define the FDTC for a braid \cite{M}.  More generally, as we discussed in \cite{ik2} and \cite{ik-qveer} and will review in Section~\ref{secFDTC}, the FDTC $c(\phi,\cL,C)$ can be defined for a closed braid $\cL$ in an open book $(S,\phi)$.

Although the FDTC is originated from very simple geometric idea, it plays a fundamental role in the study of (contact) 3-manifolds. Here we give an incomplete list of properties of the FDTC that are relevant to this paper and demonstrate the usefulness of the FDTC. 

Let $(M_{(S,\phi)},\xi_{(S,\phi)})$ be a contact 3-manifold supported by an open book $(S, \phi)$.
According to the Giroux correspondence \cite{gi}, the contact structure $\xi_{(S,\phi)}$ is unique up to isotopy. 
For simplicity, we assume that the boundary $\partial S$ is connected and denote $c(\phi):=c(\phi, \partial S)$ and $c(\phi,\cL):=c(\phi,\cL,\partial S)$.

\begin{itemize}
\item[(i)] If $c(\phi)<0$ then $\xi_{(S,\phi)}$ is overtwisted  \cite{HKM}.
\item[(ii)] If $c(\phi)\geq 1$ and $\phi$ is pseudo-Anosov, $\xi_{(S,\phi)}$ is a perturbation of a co-oriented taut foliation of $M_{(S,\phi)}$. Hence, it is universally tight and symplectically fillable \cite{HKM2}.
\item[(iii)] If $c(\phi)=k\slash p$ and $\phi$ is pseudo-Anosov such that $k\geq 2$ and $p$ is the number of prongs at the boundary, then $\xi_{(S,\phi)}$ is universally tight \cite{CH}.
\item[(iv)] If $|c(\phi)| > 1$ then $M_{(S,\phi)}$ is hyperbolic (resp.~Seifert fibered, toroidal) if and only if $\phi$ is pseudo-Anosov (resp.~periodic, reducible) \cite{ik2}. 
\item[(v)] $|c(\phi)|$ gives rise to an estimate of the rank of the reduced Heegaard Floer homology of $M_{(S,\phi)}$  \cite{HM}.
\item[(vi)] $|c(\phi,\cL)|$ gives a lower bound of the genus of the closed braid $\cL$ in $(S, \phi)$ \cite{ik2}.
\item[(vii)] The FDTC of a braid $\in B_n$ is a certain slope of the homogenization of the Heegaard Floer $\Upsilon$-invariant \cite{FH}.
\end{itemize}

The aim of this paper is to study behavior of the FDTC under a branched covering that a basic operation in topology. 
In dimension 3, every closed, connected, orientable 3-manifold is a simple branched covering of $S^3$ branched along a link \cite{Alexander, Hi, Mo}.

In Section \ref{section:branch} we introduce an open book branched covering 
$$
\pi_\cL:(\widetilde{S}, \widetilde{\phi}) \rightarrow (S,\phi)
$$ 
branched along a closed braid $\cL$ in the open book $(S,\phi)$. The branched covering $\pi_\cL$ gives rise to a contact branched covering 
\[ 
\pi:(M_{(\widetilde{S},\widetilde{\phi})},\xi_{(\widetilde{S},\widetilde{\phi})}) \rightarrow (M_{(S,\phi)},\xi_{(S,\phi)})
\]
of the corresponding contact 3-manifold branched along the transverse link represented by the closed braid $\cL$.
The lift $\widetilde \cL:= \pi^{-1}(\cL)$ is a closed braid in the upstair open book $(\widetilde{S},\widetilde{\phi})$.
Therefore, an open book branched covering is a mapping-class-group-theoretic formulation of a contact branched covering.

Our main theorem stated below gives a relation between the FDTCs of three mapping classes involved in the branched covering; namely, the upstair monodromy $\wphi$, the lifted braid $\widetilde\cL$, and the downstair braid $\cL$. 
We assume the branched covering is \emph{fully ramified}; that is, near the branch set the covering is non-trivial (Definition \ref{definition:fully-ramified}).

\begin{restatable*}{theorem}{Coveringtheorem}
\label{theorem:covering}
Let $\pi_\cL:(\widetilde{S},\widetilde{\phi}) \rightarrow (S,\phi)$ be a fully ramified branched covering of an open book $(S,\phi)$, branched along a closed braid $\cL$ in the open book $(S,\phi)$. Assume that $\chi(\widetilde S) < 0$. 
For a boundary component $C$ of $S$ let $\widetilde{C}$ be a connected component of the preimage $\pi^{-1}(C)$. Let $d(\pi, \widetilde{C})$ denote the degree of the covering $\pi|_{\widetilde{C}}: \widetilde{C} \rightarrow C$.
Then we have 
\begin{equation}
\label{eq:key}
c(\widetilde{\phi},\widetilde{C}) = c(\widetilde{\phi},\widetilde{\cL}, \widetilde{C}) =  c(\phi,\cL,C)/d(\pi,\widetilde{C}).
\end{equation}
\end{restatable*}

The first equality, which turns out to be a consequence of fully ramified assumption, is a bit surprising, because 
in the downstair there is no relation between the monodromy $\phi$ of the open book $(S, \phi)$ and the closed braid $\widetilde\cL$ in the same open book. 

For example, consider the disk open book $(D, \id)$ for the 3-sphere. The FDTC of the monodromy is trivial; $c(\id,\partial D)=0$. On the other hand, a closed braid $\cL$ in the open book $(D, \id)$ is the braid closure of an element $\beta \in B_n=\MCG(D_n)$; i.e., $c(\id,\cL,\partial D) = c(\beta,\partial D)$ which can be arbitrary large.

Theorem \ref{theorem:covering} has numerous applications to both topology and contact topology which we discuss in Sections \ref{section:right-veering} and \ref{section:applications}.

In Theorem \ref{theorem:reducible-right-veering}, we characterize right-veering for reducible mapping classes in terms of the FDTC and the interior FDTC that we define. Then we prove that right-veeringness is preserved under fully ramified branched coverings (Corollary~\ref{cor:covering-right-veering}).

In Corollary~\ref{cor:overtwisted} we show that the geometric type of the complement of the branch locus is carried over to the branched cover.

\begin{restatable*}{corollary}{Overtwistedcor}
\label{cor:overtwisted}
Let $(\widetilde{S},\widetilde{\phi})$ be a fully ramified $d$-fold branched covering of $(S,\phi)$ branched along a closed braid $\cL$ with $\chi(\widetilde S) < 0$. 
Assume that 
\begin{enumerate}
\item[(a)] both $\partial S$ and $\partial \widetilde{S}$ are connected and $|c(\phi,\cL,\partial S)| > d,$ or
\item[(b)] $|c(\phi,\cL,C)| > 4 d(\pi,\widetilde{C})$ for every boundary component $C \subset \partial S$ and  connected component $\widetilde{C} \subset \pi^{-1}(C)$.
\end{enumerate}
If $\cL$ is Seifert-fibered $($resp.~toroidal, hyperbolic$)$ then  $M_{(\widetilde{S},\widetilde{\phi})}$ is Seifert-fibered $($resp.~toroidal, hyperbolic$).$
\end{restatable*}

As an application to knot theory, we can detect non-quasi-alternating links.

\begin{restatable*}{corollary}{QAcor}
\label{cor:quasialternating}
Let $\beta \in B_{2n+1}= \MCG(D_{2n+1})$ be a pseudo-Anosov braid. If $|c(\beta,\partial D_{2n+1})|>2$, then the closure of $\beta$ is not quasi-alternating.
\end{restatable*}

A contact 3-manifold is \emph{universally tight} if its universal covering is tight. Otherwise, since the fundamental group of a 3-manifold is residually finite, it is \emph{virtually overtwisted}; in other words, it has a finite covering which is overtwisted. In the study of Legendrian and transverse links, analogous notions \emph{universally non-loose} and \emph{virtually loose} can be considered (see Definition~ \ref{def:univloose}). Despite the importance of the universally non-loose vs. virtually loose dichotomy (see Etnyre and Ng's problem list \cite{EN} that contains several related questions), these properties have been little studied.

Among the applications of Theorem \ref{theorem:covering} to contact topology, the following corollary relates right-veering braids and universally non-loose links.

\begin{restatable*}{corollary}{VLcor}
\label{cor:virtually-loose2}
Let $M$ be a closed 3-manifold without $S^{1}\times S^{2}$-summand.
Let $\mathcal{T}$ be a transverse link in a contact 3-manifold $(M,\xi)$. 
If $\mathcal{T}$ is universally non-loose then every closed braid representative of $\mathcal{T}$ with respect to every open book decomposition of $(M,\xi)$ is right-veering.
\end{restatable*}

Corollary~\ref{cor:virtually-loose2} implies that virtually loose transverse links in a tight contact 3-manifold are abundant. In particular, it negatively answers the transverse link version of Etnyre and Ng's Questions 63 and 65 in \cite{EN}; Is a  transverse link in the standard tight contact 3-sphere $(S^3, \xi_{std})$ realizing  maximal self-linking number universally non-loose (or, is the infinite cyclic covering of its complement tight)? 

\begin{restatable*}{corollary}{VLappcor}
\label{cor:virtually-loose-application}
There is a transverse link $\mathcal{T}$ in $(S^3, \xi_{std})$ which realizes the maximal self-linking number of the topological link type and is virtually loose. 
Indeed, the double cover of its complement is overtwisted.
\end{restatable*}

We also give a potential counter example of transverse link version of Question 64 in \cite{EN} on universally non-looseness in Remark \ref{rem:Question64}.

Although tight contact 3-manifolds are supported by right-veering open books \cite{HKM}, non-loose transverse links are not necessarily represented by right-veering braids. In \cite{ik-qveer} a new property of closed braids called \emph{quasi-right-veering} is introduced and it is shown that the quasi-right-veering property characterizes non-loose transverse links.
This poses a natural question:\\

\textit{What is the property of transverse links that corresponds to the right-veering-ness of closed braids?}\\

Corollary~\ref{cor:virtually-loose2} suggests the answer would be ``universally non-loose''
(see Question \ref{question:main}).

\section*{Acknowledgement}
The authors would like to thank John Etnyre for many useful comments.  They also thank the anonymous referee for his/her numerous suggestions and pointing out an error in an early version.
TI was partially supported by JSPS KAKENHI grant number 15K17540 and 16H02145.
KK was partially supported by NSF grants DMS-1206770 and DMS-2005450 and Simons Foundation Collaboration Grants for Mathematicians.

\section{Preliminaries}\label{section:closed-braid}

In this section, we recall the necessary background materials on mapping class groups, closed braids in abstract open books, the distinguished monodromy of a closed braid, and the FDTC of a closed braid.  For more details we refer the reader to \cite{ik-qveer}.

\subsection{Open books and contact 3-manifolds}

Let $S = S_{g,d}$ be an oriented compact surface with genus $g$ and $d$ boundary components.  Throughout the paper we assume that $d>0$. 
Let $P = \{p_1,\ldots,p_n\}$ be a (possibly empty) set of $n$ distinct interior points of $S$. We denote by $\Diff(S, P, \partial S)$ (or, $\Diff(S, \partial S)$ if $P$ is empty) the group of orientation preserving diffeomorphisms of $S$ that fix $P$ setwise and the boundary $\partial S$ pointwise. 
The mapping class group $\MCG(S,P)$ (denoted by $\MCG(S)$ if $P$ is empty) is the group of isotopy classes of diffeomorphisms in $\Diff(S, P, \partial S)$. 
We will often view the marked points in $P$ as punctures, and we identify $\MCG(S,P)$ with the mapping class group $\MCG(S\setminus P)$ of a punctured hyperbolic surface $S\setminus P$.
We may denote the mapping class of a diffeomorphism $\vphi$ by $[\vphi]$. 

By ``\emph{abstract open book}'' \cite{E} we mean a pair $(S,\varphi)$ of surface $S$ and diffeomorphism $\varphi \in \Diff(S,\partial S)$ such that $\varphi$ is the identity in a neighborhood $\nu(\partial S)$ of $\partial S$, and by ``\emph{open book}'' we mean a pair $(S,\phi)$ where $\phi \in \MCG(S)$.

An abstract open book $(S,\varphi)$ can be used to define a 3-manifold $M_{(S,\varphi)}$ as a quotient space:
\begin{equation}
\label{eqn:open-book-manifold} 
M_{(S, \varphi)} := S \times[0,1] \slash \sim , \quad \begin{cases}
(x,1)\sim(\varphi(x),0) & \mbox{for } x \in S,\\
(x,t)\sim (x,s) & \mbox{for } x \in \partial S \mbox{ and }  t,s \in [0,1].
\end{cases}
\end{equation} 
Let $\mathcal Q: S\times[0,1] \rightarrow M_{(S,\varphi)}$ denote the quotient map. Call 
$B:=\mathcal Q (\partial S \times \{t\})$
the \emph{binding} and 
$S_t:=\mathcal Q(S\times\{t\})$
a \emph{page}. 
The pages $S_0$ and $S_1$ are identified in $M_{(S, \vphi)}$ via the diffeomorphism $\varphi$. 

Let $pr:S\times[0,1) \to S; \ (x, t)\mapsto x$ be a projection.  
For a point ${\bf x} \in M_{(S, \vphi)}\setminus B$, the preimage $Q|_{S\times[0, 1)}^{-1}({\bf x})$ is a single point, say $(x, t)$. 
Whereas, if ${\bf x} \in B$, the preimage $Q|_{S\times[0, 1)}^{-1}({\bf x})=\{x\} \times[0,1)$ is an interval for some unique $x \in \partial S$.  
With this in mind we define 
$$p: M_{(S, \vphi)}\to S; \ \ \ {\bf x} \mapsto pr \circ (\mathcal Q|_{S\times[0, 1)}^{-1}({\bf x}))= x.$$

\subsection{Closed braids and distinguished monodromies}
\label{section:braid}
 
A \emph{closed $n$-braid $L$ with respect to an abstract open book $(S,\varphi)$} is an oriented link in $M_{(S,\varphi)} \setminus B$ such that $L$ intersects every page $S_t$ positively and transversely at $n$ points. If we want to emphasize the underlying abstract open book, $L$ is denoted by the triple $((S,\varphi),L)$. 
We will allow the braid $L$ to be empty, i.e., $L=\emptyset$. 
Let $P:=p(L \cap S_0) \subset S$
be the set of $n$ points where $L$ intersects the page $S_0$. In general we have $\vphi(P) \neq P$.

If $\vphi(P) \neq P$, we can always find a closed braid $L'=((S,\varphi),L')$ that is braid isotopic to $L$ and $\vphi(P') = P' := p(L' \cap S_0)$. 
For example, choose a set $P'$ of $n$ points in the neighborhood $\nu(\partial S)$ of $\partial S$ on which $\vphi=id$, then isotope $L$ so that the resulting $L'$ intersects the page $S_0$ at $P'$. 
It is easy to see that there are infinitely many choices for such $L'$. In the rest of the paper we will assume the following.

\begin{assump}
Every closed braid $((S,\varphi),L)$ with $P:=p(L \cap S_0)$ has: 
\begin{equation}
\label{eqn:assumptionweak} \varphi(P)=P \quad \mbox{(as sets) }
\end{equation}
\end{assump}

With this assumption, a closed braid $L$ gives rise to an element of the $n$-stranded surface braid group $\pi_1(C(S,P))$ where $C(S, P)$ is the configuration space of $n$ distinct, unordered points in $S$.  
The Generalized Birman Exact Sequence Theorem \cite{FM} states that the sequence 
$$1 \longrightarrow \pi_1(C(S,P)) \stackrel{\mathrm{Push}}{\longrightarrow} \MCG(S, P) \stackrel{\mathrm{Forget}}{\longrightarrow} \MCG(S) \to 1$$ is exact. 
Thus there is a diffeomorphism $\varphi_{L} \in \Diff(S,P,\partial S)$ 
such that $[\vphi_L]=(\mathrm{Push}(L)) [\vphi]$ and $\mathrm{Forget}([\vphi_{L}])=[\vphi]$.
Construction of $\vphi_{L}$ from $\vphi$ and $L$ is given in \cite[Section 2.3]{ik-qveer}. When $L=\emptyset$ we set $\vphi_L=\vphi$.

Define 
\begin{equation}
\label{eqn:distinguished} 
(M_{(S,\varphi_{L})}, L_{(P, \vphi_L)}) := \left( (S, P) \times [0,1] \right) / \sim_{\vphi_L} 
\end{equation} 
where the equivalence relation $\sim_{\vphi_L}$ is defined by 
$$
\begin{cases}  
(x,1)\sim_{\vphi_L}(\varphi_{L}(x),0) & \mbox{for } x \in S,\\
(x,t)\sim_{\vphi_L} (x,s) & \mbox{for } x \in \partial S \mbox{ and }  t,s \in [0,1].
\end{cases}
$$
Note that $L_{(P, \vphi_L)}$ is a closed braid in the abstract open book $(S,\varphi_{L})$.
By (\ref{eqn:assumptionweak}) both $\vphi$ and $\vphi_{L}$ are in $\Diff(S, P, \partial S)$ and in general $\vphi\neq \vphi_{L}$ as they permute the points in the set $P$ differently. However, since $\vphi_{L}$ and $\vphi$ are isotopic within $\Diff(S,\partial S)$, an isotopy between $\vphi_{L}$ and $\vphi$ gives a page-preserving diffeomorphism between $M_{(S, \vphi_{L})}$ and $M_{(S, \vphi)}$.
With this diffeomorphism we get 
$$(M_{(S,\varphi_{L})}, L_{(P, \vphi_L)}) \simeq (M_{(S, \vphi)}, L).$$
In the rest of the paper we often denote $L_{(P, \vphi_L)}$ by $L$ for notational simplicity if no confusion arises.

\begin{definition}
Let $((S,\varphi),L)$ be a closed braid that satisfies  (\ref{eqn:assumptionweak}). 
The mapping class $[\varphi_L] \in \MCG(S,P)$ is called 
the \emph{distinguished monodromy} of $((S,\varphi),L)$. 
\end{definition}

In Theorem~\ref{theorem:distinguished-monodromy} below, we will discuss uniqueness of the distinguished monodromy.  
Namely, if $L$ and $L'$ are braid isotopic then their distinguished monodromies are the same up to a point-changing isomorphism, which we define shortly. 

For abstract open books $(S,\varphi)$ and $(S',\varphi')$ with $S=S'$ and $[\varphi]=[\varphi'] \in \MCG(S)$, an isotopy between $\varphi$ and $\varphi'$ yields a diffeomorphism
$\rho: M_{(S,\varphi)} \rightarrow M_{(S',\varphi')}$ 
that preserves each page; $\rho(S_t)=S'_t$. 
See \cite[Section 3]{ik-qveer} for detail. 
With this $\rho$ we define an equivalence relation below: 

\begin{definition}\label{def:equivalent-braids}
Two closed braids $((S,\varphi),L)$ and $((S',\varphi'),L')$ are called \emph{equivalent}
if $S=S'$, $[\varphi]=[\varphi'] \in \MCG(S)$, and $\rho(L)$ is isotopic to $L'$ thorough closed braids in $(S',\varphi')$. 
\end{definition}

To relate the distinguished monodromies of equivalent closed braids, we use the following point-changing isomorphism:

\begin{definition}\label{def:point-changing}
Let $P$ and $P'$ be finite sets of points in $S$ with the same cardinality.
A \emph{point-changing isomorphism} is an isomorphism 
$$
\Theta:\MCG(S,P) \rightarrow \MCG(S,P') 
$$ 
defined by $\Theta([\psi])= [\theta\circ \psi \circ \theta^{-1}]$, where $\theta \in \Diff(S, \partial S)$ is an orientation-preserving diffeomorphism isotopic to $id_S$ such that $\theta(P)=P'$.
\end{definition}

When $P=P'$, a point-changing isomorphism is an inner automorphism of $\MCG(S,P)$. In this sense ``up to point-changing isomorphism'' can be understood as a generalization of ``up to conjugation'' to the case where marked point sets are different. 

The following theorem is proved in \cite[Proposition 2.7,  Theorem 2.9]{ik-qveer}.

\begin{theorem}
\label{theorem:distinguished-monodromy}
Let $((S,\varphi),L)$ and $((S,\varphi'),L')$ be equivalent closed braids that satisfy the condition (\ref{eqn:assumptionweak}) for $P:=p(S_0 \cap L)$ and $P':=p(S_0\cap L')$. 
Then there is a point-changing isomorphism $\Theta:\MCG(S,P) \rightarrow \MCG(S,P')$ such that $\Theta([\varphi_L])=[\varphi'_{L'}]$. 
$($Such $\Theta$ is not unique in general.$)$
\end{theorem}

\subsection{The FDTC of a closed braid}\label{secFDTC}

Let $((S, \vphi), L)$ be a closed braid with the distinguished monodromy $[\vphi_L]$. 
Viewing the mapping class group $\MCG(S, P)$ of a {\em marked} surface as the mapping class group $\MCG(S \setminus P)$ of a hyperbolic {\em punctured} surface, we can define the Nielsen-Thurston type of $[\vphi_L]$. 
To define right-veeringness and the FDTC of $[\vphi_L]$ we adopt Honda-Kazez-Mati\'c's definitions in \cite{HKM}. 
By the nature of those definitions, we see that the Nielsen-Thurston type, right-veeringness, and the FDTC are preserved under a point-changing isomorphism.  
Therefore, by Theorem~\ref{theorem:distinguished-monodromy} we have the following.

\begin{proposition}\label{prop:FDTCandRV}
If $((S,\varphi),L)$ and $((S,\varphi'),L')$ are equivalent closed braids then for each boundary component $C \subset \partial S$ 
\begin{enumerate}
\item
the FDTCs are equal: $c([\varphi_L],C)=c([\varphi'_{L'}],C),$ 
\item
$[\varphi_L]$ is right-veering if and only if $[\varphi'_{L'}]$ is right-veering.
\end{enumerate}
\end{proposition}

\section{Branched coverings of open books} 
\label{section:branch}

In this section we discuss branched coverings of open books along  closed braids.
In Section~\ref{sec31} we define an open book branched covering. 
In Section~\ref{sec:32} we compare contact branched covering and open book branched covering.
In Section~\ref{sec:33} we define an equivalence relation on branched coverings and study its properties. 
In Section~\ref{sec:34} we list various equivalence classes we use in Sections 4 and 5.

Note that Casey \cite[Theorem 3.4.1]{C} has studied branch coverings in the special case where $(S, \phi)=(D^2, id)$ and $(M, \xi)=(S^3, \xi_{std})$. See also Giroux's work \cite[Corollarie 5]{gi}. 
Throughout the paper, covering spaces are assumed to be connected.

\subsection{Definition of an open book branched covering}\label{sec31}

\begin{definition}
\label{definition:branch-openbook}
Let $P \subset S$ be a set of $n$ interior points. 
Let $$\pi: (\widetilde S, \widetilde P)\to (S, P)$$ be a branched covering of $S$ with branch set $P$ and $\widetilde{P} :=\pi^{-1}(P)$. 
Let $L=((S, \vphi), L)$ be a closed braid with $P=p(L\cap S_0)$. 
We say that an abstract open book $(\widetilde{S},\widetilde{\varphi_L})$ is \emph{an open book branched covering of $(S, \varphi)$ along $L$} and denote it by 
$$
\pi_L: (\wS, \wvphiL)\to(S, \vphi),
$$ 
if 
\begin{enumerate}
\item
$\wvphiL(\widetilde{P})=\wP$ and 
\item
the following diagram commutes; i.e., 
\begin{equation}\label{eq:commutative}
\pi \circ \wvphiL = \varphi_L \circ \pi. 
\end{equation}
\begin{equation*}
\xymatrix{
(\widetilde{S}, \widetilde{P}) \ \ar[r]^{\wvphiL} \ar[d]_{\pi}& (\widetilde{S}, \widetilde{P})  \ar[d]^{\pi}\\(S, P) \ar[r]^{\varphi_{L}}& (S, P)
} 
\end{equation*}
\end{enumerate}
\end{definition}

\begin{remark}
We will view $\wvphiL\in\Diff(\wS, \wP, \partial \wS)$ and $[\wvphiL]\in \MCG(\wS, \wP)$. 
When we forget the marked points in $\wP$ and see $\wvphiL$ as an element of $\Diff(\wS, \partial \wS)$, its mapping class is denoted by $\f[\wvphi]:=\f([\wvphi]) \in\MCG(\wS)$ where $\f: \MCG(\wS, \wP)\to\MCG(\wS)$ is the forgetful map.
\end{remark}

By the commutativity (\ref{eq:commutative}) the map $\pi\times id: \widetilde S\times[0,1] \to S\times[0,1]$  
induces a branch covering 
\[ 
\pi: M_{(\widetilde S, \wvphi)} \to M_{(S, \vphi_L)}
\] 
of the 3-manifold along the closed braid $L$ (we abuse the notation $\pi$ here).
The lift $\widetilde{L}:=\pi^{-1}(L)$ of $L$ is a closed braid in the abstract open book $(\widetilde{S},\wvphi)$. 
In terms of the quotient space construction (\ref{eqn:distinguished}), $\widetilde L=\widetilde P \times [0,1] / \sim_{\wvphi}$.
Since $\wvphiL(\widetilde{P})=\wP$ 
the assumption (\ref{eqn:assumptionweak}) is satisfied for $\widetilde L$ and $[\wvphi]$ serves as the distinguished monodromy of $\widetilde L$.

\subsection{Contact branched coverings}
\label{sec:32}

For an abstract open book $(S, \vphi)$, Thurston and Winkelnkemper's construction \cite{TW} gives a contact structure $\xi=\xi_{(S, \vphi)}$ on the manifold $M_{(S, \varphi)}$ that is {\em supported by} $(S, \vphi)$. That is, there exists a contact 1-form $\alpha$ such that $\xi=\ker \alpha$, $d\alpha$ is an area form of every page $S_t$ and $\alpha>0$ on the binding $B$. A contact structure supported by $(S,\varphi)$ is unique up to isotopy due to Giroux \cite{gi}.  

In \cite[Theorem 3.2]{Pav}, Pavelescu shows that given a transverse link $\mathcal T$ in a contact manifold $(M_{(S, \varphi)}, \xi_{(S, \varphi)})$, we can always find a braid, $((S, \varphi), L)$, that is transversely isotopic to $\mathcal T$. We say that the transverse link $\mathcal T$ is {\em represented by} a closed braid $((S, \vphi), L)$.

\begin{lemma}\label{lem:T and L}
Assume closed braids $((S, \vphi), L)$ and $((S', \vphi'), L')$ are equivalent. Then a transverse link $\mathcal T$ is represented by $((S, \vphi), L)$ if and only if $\mathcal T$ is represented by $((S', \vphi'), L')$.
\end{lemma}

With this in mind, we study a contact branched covering along $\mathcal T$ through an open book branched covering along $L$.

For any branched covering $\Pi: \widetilde M \to M$ along $\mathcal{T}$, the differential $d\Pi$ has $\dim({\rm im}\ d\Pi)=1$ along the branch locus $\widetilde{\mathcal T}:= \Pi^{-1}(\mathcal T)$. 
This means for any contact structure $\eta$ on $\widetilde M$ the pushforward $\Pi_*(\eta)$ cannot be a contact structure on $M$. 
However, in \cite{Geiges} Geiges shows that $\widetilde M$ admits a contact structure $\widetilde \xi$ whose pushforward {\em outside} a neighborhood of $\widetilde{\mathcal T}$ coincides with $\xi$ and moreover $\widetilde{\mathcal T}$ is a transverse link with respect to $\widetilde\xi$.   

\begin{definition}
For the above branched covering $\Pi: \widetilde M \to M$, we say that $(\widetilde M, \widetilde\xi)$ is a {\em contact branched covering} of $(M, \xi)$ along $\mathcal T$ 
and denoted by 
$$\Pi: (\widetilde{M},\widetilde{\xi})\to(M, \xi)$$ abusing the notation of $\Pi$. Such a contact structure $\widetilde \xi$ is unique up to isotopy \cite{ON}. 
\end{definition}

We want to relate contact branched coverings and open book branched coverings. 
In the next definition, we will extend the notion of `supported by' to these branched coverings. 

\begin{definition}\label{def:support-extension}
Let $\Pi: (\widetilde{M},\widetilde{\xi})\to(M, \xi)$ be a contact branched covering along a transverse link $\mathcal T$. 
Denote $\widetilde{\mathcal T}:=\Pi^{-1}(\mathcal T)$. 

Let $\pi_L: (\wS, \wvphiL)\to(S, \vphi)$ be an open book branched covering along a closed braid $L$. 
Let $\pi: M_{(\widetilde S, \wvphi)} \to M_{(S, \vphi_L)}$ be the induced branched covering, and denote $\widetilde{L}:=\pi^{-1}(L)$.

We say that $\Pi$ is {\em supported by} $\pi_L$ if there exist contactomorphisms 
$\theta: (M,\xi) \rightarrow (M_{(S,\varphi)},\xi_{(S,\varphi)})$ 
and
$\widetilde{\theta}: (\widetilde{M},\widetilde{\xi}) \rightarrow (M_{(\wS, \wvphi)},\xi_{(\wS,\wvphi)})$ such that  
\begin{enumerate}
\item
$\theta(\mathcal{T})=L$, 
\item
$\widetilde{\theta}(\widetilde{\mathcal{T}})=\widetilde{L}$,  and 
\item
the following diagram commutes away from the branch loci:
\[
\xymatrix{
(\widetilde{M},\widetilde{\xi}) \ar[r]^-{\widetilde{\theta}} \ar[d]_{\Pi}& (M_{(\wS, \wvphi)},\xi_{(\wS,\wvphi)}) \ar[d]^{\pi}\\ 
(M,\xi) \ar[r]^-{\theta}& (M_{(S,\varphi)}, \xi_{(S,\vphi)})
} \]
\end{enumerate}
\end{definition}

\subsection{Equivalent branched coverings of open books}
\label{sec:33}

We introduce an equivalence relation on the set of open book branched coverings and study its properties.  
\begin{definition}\label{def:equivalence2}
Let $((S, \vphi_0), L_0)$ and $((S, \vphi_1), L_1)$ be equivalent closed braids satisfying \ref{eqn:assumptionweak}).
For $i=0, 1$, let $P_i := p(L_i \cap S_0)$. 
Fix $$\pi_i: (\wS_i, \widetilde P_i) \to (S, P_i)$$ a branched covering with branch set $P_i$. 
Associated to $\pi_i$, assume that $(S, \vphi_i)$ admits an open book branched covering 
$$\pi_{L_i}:= \ {\pi_i}_{L_i}:(\wS_i, \widetilde{{\vphi_i}_{L_i}})\to (S, \vphi_i)$$ 
along $L_i$ as introduced in Definition~\ref{definition:branch-openbook} (we denote $\pi_{L_i}={\pi_i}_{L_i}$ for simplicity).  
We say that the branched coverings $\pi_{L_0}$ and $\pi_{L_1}$ are \emph{equivalent} if there exist diffeomorphisms $g:(S, P_0)\to(S, P_1)$ and $\widetilde g: (\wS_0, \widetilde P_0)\to(\wS_1, \widetilde P_1)$ such that $g$ is isotopic to $id_S$ if we forget the marked points and the following diagram commutes. Namely: 
\begin{eqnarray*}
\mbox{(Sides) } && \pi_i \circ \widetilde{{\varphi_i}_{L_i}} = {\varphi_i}_{L_i} \circ \pi_i\\
\mbox{(Front and Back) } && \pi_1 \circ \widetilde g = g \circ \pi_0 \\
\mbox{(Top) } && \widetilde{{\vphi_1}_{L_1}} \circ \widetilde g = \widetilde g \circ \widetilde{{\vphi_0}_{L_0}} \\
\mbox{(Bottom) } && {\vphi_1}_{L_1} \circ g = g \circ {\vphi_0}_{L_0}
\end{eqnarray*}
\begin{equation}
\label{eqn:cdiagram}
\xymatrix{ 
 & (\widetilde{S}_0,\widetilde{P}_0) \ar[ld]_{\widetilde{{\varphi_0}_{L_0}}} \ar[dd]_(.4){\pi_0} \ar[rr]^{\widetilde{g}} &  & (\widetilde{S}_1,\widetilde{P_1}) \ar[ld]^{\widetilde{{\varphi_1}_{L_1}}}\ar[dd]^(.4){\pi_1}\\
(\widetilde{S}_0,\widetilde{P}_0) \ \ar[rr]_(.6){\widetilde g} \ar[dd]_{\pi_0}& & (\widetilde{S}_1,\widetilde{P}_1) \ar[dd]_(.3){\pi_1} &\\
& (S,P_0) \ar[ld]_{{\varphi_0}_{L_0}} \ar[rr]^(.6){g}&  &(S,P_1) \ar[ld]^{{\varphi_1}_{L_1}}\\
 (S,P_0) \ar[rr]^{g}& &(S,P_1)&
}
\end{equation}
\end{definition}

Below are properties of the equivalence relation. 

\begin{proposition}
\label{prop:well-defined}
Assume that open book branched coverings $\pi_{L_i}: (\widetilde S_i, \widetilde{{\vphi_i}_{L_i}})\to (S, \vphi_i)$ where $i=0, 1$ are equivalent in the sense of Definition~\ref{def:equivalence2}.
\begin{enumerate}
\item[(i)] The diffeomorphisms $g:(S, P_0)\to(S, P_1)$ and $\tilde g: (\wS_0, \widetilde P_0)\to(\wS_1, \widetilde P_1)$ in (\ref{eqn:cdiagram}) naturally give page-preserving diffeomorphisms of 3-manifolds 
$$
G: M_{(S, \, {\vphi_0}_{L_0})} \to M_{(S, \, {\vphi_1}_{L_1})}
$$ 
and 
$$
\widetilde G: M_{(\wS_0, \, \widetilde{{\vphi_0}_{L_0}})} \to M_{(\wS_1, \, \widetilde{{\vphi_1}_{L_1}})}
$$
such that $G(L_0)= L_1$ and $\widetilde G(\widetilde L_0)=\widetilde L_1$, and the following diagram commutes. 
$$
\xymatrix{
M_{(\wS_0,\, \widetilde{{\vphi_0}_{L_0}})} \ \ar[r]^{\widetilde{G}} \ar[d]_{\pi_0}& M_{(\wS_1,\ \widetilde{{\vphi_1}_{L_1}})}  \ar[d]^{\pi_1}\\M_{(S, \ {\vphi_0}_{L_0})} \ar[r]^{G}& M_{(S, \ {\vphi_1}_{L_1})}
}
$$
\item[(ii)]
Moreover, for $i=0,1$ we can find a contact structure $\xi_i$ on $M_{(S, \, {\vphi_i}_{L_i})}$ supported by $(S, {\vphi_i}_{L_i})$ so that $G$ becomes a contactomorphism: 
$$
G:(M_{(S, \, {\vphi_0}_{L_0})},\ \xi_0) \to (M_{(S, \, {\vphi_1}_{L_1})}, \ \xi_1)
$$
Let 
$$
\pi_i:(M_{(\wS_i, \, \widetilde{{\vphi_i}_{L_i}})}, \ \widetilde\xi_i)\to (M_{(S, \, {\vphi_i}_{L_i})}, \ \xi_i)
$$ 
be a contact branched covering along $L_i$ such that the pushforward ${\pi_i}_*\widetilde\xi_i$ coincides with $\xi_i$ away from the branch locus. 
With a small perturbation of $\widetilde\xi_i$ near the branch locus, if needed,  
\begin{itemize}
\item
$\widetilde\xi_i$ is supported by $(\wS_i, \widetilde{{\vphi_i}_{L_i}})$
\item
$\widetilde G$ becomes a contactomorphism, and 
\item
the following diagram commutes away from the branch loci. 
\end{itemize}
\begin{equation*}\label{eq:diagramG}
\xymatrix{
(M_{(\wS_0,\, \wvphi_0)}, \ \widetilde \xi_0 ) \ \ar[r]^{\widetilde{G}} \ar[d]_{\pi_0}& (M_{(\wS_1,\, \wvphi_1)},\ \widetilde \xi_1)  \ar[d]^{\pi_1}\\
(M_{(S, \, {\vphi_0}_{L_0})},\ \xi_0) \ar[r]^{G}& (M_{(S, \, {\vphi_1}_{L_1})}, \ \xi_1 )
}
\end{equation*}

\item[(iii)]
The distinguished monodromies $[\widetilde{\vphi_{L_0}}] \in \MCG(\wS_0,\wP_0)$ of $\widetilde{L_0}$ and $[\widetilde{\vphi_{L_1}}] \in \MCG(\wS_1,\wP_1)$ of $\widetilde{L_1}$ are related by the point-changing isomorphism $$\tilde{\Theta}: \MCG(\wS_0,\wP_0) \rightarrow \MCG(\wS_1,\wP_1)$$ defined by $\tilde{\Theta}([m]) =[\widetilde{g}\circ m \circ\widetilde{g}^{-1}]$. 

\item[(iv)]
The distinguished monodromies $[\widetilde{\vphi_{L_0}}]$ and $[\widetilde{\vphi_{L_1}}]$ have exactly the same 
FDTC, the right-veering property, and the Nielsen-Thurston type. 
\end{enumerate}
\end{proposition}

\begin{proof}

(i) Recall that for an abstract open book $(S,\vphi)$, the corresponding 3-manifold $M_{(S,\vphi)}$ is a quotient space of $S\times[0,1]$ by (\ref{eqn:open-book-manifold}). Let us consider diffeomorphisms $g\times id: S \times [0,1] \rightarrow S\times[0,1]$ and $\widetilde{g}\times id: \wS \times [0,1] \rightarrow \wS\times[0,1]$. By commutativity of the diagram (\ref{eqn:cdiagram}), $g$ and $\widetilde{g}$ induce desired diffeomorphisms $G$ and $\widetilde{G}$.

(ii) Choose a contact structure $\xi_0$ that is supported by $(S, {\vphi_0}_{L_0})$ (eg. apply the Thurston-Winkelnkemper construction \cite{TW}) and define $\xi_1:= G_*\xi_0$.
The rest of the statement follows from (i) and Geiges' construction \cite{Geiges}.

(iii) It follows from the commutative diagram (\ref{eqn:cdiagram}).

(iv) It is a consequence of (iii), because as discussed in Section~\ref{secFDTC} the FDTC, the right-veering property and the Nielsen-Thurston type are invariant under a point-changing isomorphism.
\end{proof}

\subsection{Notations and definitions of equivalence classes}
\label{sec:34}

In the rest of the paper, we adopt the following notations based on the study of equivalence relations in the above sections.  

\begin{itemize}
\item
Let $\vphi\in\Diff(S, \partial S)$. 
The mapping class of $\vphi$ is denoted by $\phi:=[\vphi]\in\MCG(S)$. 

\item
The pair $(S, \phi)$ is called an {\em open book} (without the adjective `abstract') and we say that $(S, \phi)$ is represented by $(S, \vphi)$.
\end{itemize}
In Definition~\ref{def:equivalent-braids} we defined equivalent closed braids. 
\begin{itemize}
\item
The equivalence class  of a closed braid $L=((S, \vphi), L)$ is denoted by $\cL=[((S, \vphi), L)]$. We say that $\cL$ is a {\em closed braid} in the open book $(S, \phi)$. 
\end{itemize}
The following definition is well-defined thanks to Theorem~\ref{theorem:distinguished-monodromy}. 
\begin{itemize}
\item
By the \emph{distinguished monodromy}, denoted by $\phi_\cL$, of a closed braid $\cL$ in the open book $(S, \phi)$, we mean the point-changing isomorphism class of $[\vphi_L]$.
\end{itemize}
The following are well-defined by Proposition~\ref{prop:FDTCandRV}.
\begin{itemize}
\item 
 
The {\em FDTC of the closed braid} $\cL$ in $(S,\phi)$ with respect to the boundary component $C\subset \partial S$ is defined to be the FDTC $c([\varphi_L],C)$ of the distinguished monodromy $[\varphi_L]$. 
We denote it by $c(\phi, \cL, C)$. 

\item
We say the closed braid $\cL$ is {\em right-veering} if $[\vphi_L]\in\MCG(S \setminus P)$ 
is right-veering. 
\end{itemize}
Proposition~\ref{prop:well-defined} justifies the following definition. Let
$\pi_L: (\wS, \wvphiL)\to(S, \vphi)$ be an open book branched covering and denote $\widetilde\phi:=\f[\wvphi]$ where $\f:\MCG(\wS, \widetilde P)\to \MCG(\wS)$ is the forgetful map. 
\begin{itemize} 
\item
We say that the open book $(\wS, \wphi)$ is a {\em branched covering of $(S, \phi)$ along the closed braid $\cL$,} and  
the closed braid $\widetilde{\cL}:=[((\widetilde{S},\wvphi),\widetilde{L})]$ in $(\widetilde{S}, \widetilde \phi)$ is called the \emph{lift} of $\cL$.  
\end{itemize}
The following is well defined due to Lemma~\ref{lem:T and L}. Let $\mathcal{T}$ be a transverse link in a contact 3-manifold $(M,\xi)$.
\begin{itemize} 
\item
If $\mathcal T$ is {\em represented by} a closed braid $((S, \vphi), L)$ then we say that $\mathcal{T}$ is {\em represented by} the closed braid $\cL$ in $(S,\phi)$.
\end{itemize}

\section{Formula of FDTC under branched coverings}

To state the main result (Theorem~\ref{theorem:covering}) we define fully ramified branched covers. 

\begin{definition}(cf. \cite{W}) 
\label{definition:fully-ramified}
A branched covering $\pi: (\widetilde{S}, \widetilde P) \rightarrow (S, P)$ is \emph{fully ramified} if for every branch point $\widetilde p \in \widetilde P$ there exists a disk neighborhood $\widetilde N$ containing $\widetilde p$ such that the restriction $\pi|_{\widetilde{N}}: \widetilde{N} \rightarrow \pi(\widetilde{N})$ is a non-trivial branched covering with a single branch point.
\end{definition}

Cyclic branched coverings are fully ramified. 
On the other hand, simple branched coverings, which often appear in the study of 3-manifolds (eg. Hilden-Montesinos' theorem \cite{Hi,Mo}), are not fully ramified unless they are double branched covers.

The following theorem shows that the FDTC behaves nicely under a fully ramified branched covering map.

\Coveringtheorem

There are two comments on the theorem.
\begin{enumerate}
\item
In (\ref{eq:key}) the first $c(\widetilde{\phi},\widetilde{C})$ is the FDTC of $\widetilde \phi = \f[\wvphi] \in \MCG(\widetilde S)$ where $\f:\MCG(\wS, \wP)\to\MCG(\wS)$ is the forgetful map, and the second $c(\widetilde{\phi},\widetilde{\cL}, \widetilde{C})$ is the FDTC of the distinguished monodromy $[\wvphi] \in \MCG(\widetilde S, \widetilde P)$.
They belong to different mapping class groups but their FDTCs are the same. This is because, for pseudo-Anosov case, they share the same pair of transverse measured foliations. 
For the periodic case, this is a consequence of the Euler characteristic condition $\chi(\widetilde{S})<0$.

\item
The fully ramified condition is used when $[\vphi_L]$ is pseudo-Anosov.
More specifically, for the pseudo-Anosov case only the following condition 
\begin{equation}\label{eq:condition}
\mbox{at every point } \widetilde p \in \widetilde P \mbox{ the number of prongs  is greater than one;}  
\end{equation}
is sufficient to obtain (\ref{eq:key}), and the fully ramified condition guarantees (\ref{eq:condition}). 
\end{enumerate}

\begin{proof}[Proof of Theorem~\ref{theorem:covering}]
Suppose that a closed braid $((S, \vphi), L)$ with (\ref{eqn:assumptionweak}) represents $\cL$, and $\pi_L: (\wS, \wvphi) \to (S, \vphi)$ is an open book branched covering along $L$ with $\f[\wvphi]=\wphi \in \MCG(\wS)$. 
The proof is divided into three cases, according to the Nielsen-Thurston type of $[\varphi_L]$.
\\

\noindent
{\bf Case 1:} $[\varphi_L]$ is pseudo-Anosov.\\

\noindent 
Suppose that the distinguished monodromy $[\varphi_L]$ is pseudo-Anosov. 
We think of $P$ as the set of punctures of the surface $S\setminus P$, and $\vphi_L$ and $\wvphi$ as maps: 
\begin{eqnarray*}
&\varphi_L& : S\setminus P \rightarrow S \setminus P,\\
&\widetilde{\vphi_L}& : \widetilde{S} \setminus \widetilde{P} \rightarrow  \widetilde{S} \setminus \widetilde{P}.
\end{eqnarray*}
The branched covering $\pi: (\widetilde S, \widetilde P)\to (S, P)$ can be viewed as a genuine unbranched covering $$\pi:\widetilde{S} \setminus \widetilde{P} \rightarrow S\setminus P.$$
The commutativity (\ref{eq:commutative}) still holds; namely, 
\begin{equation}\label{eq:new-commutativity}
\pi \circ \wvphi = \varphi_L \circ \pi.
\end{equation}

Let $\varphi_L^{\sf p.A}:S \setminus P \rightarrow S \setminus P$ be a pseudo-Anosov representative of $[\varphi_L] \in \MCG(S\setminus P)$. 
Since $\varphi_L^{\sf p.A}$ is only freely isotopic to $\varphi_L$, it is possible that $\varphi_L^{\sf p.A}|_{\partial S} \neq \id_{\partial S}$.  
Even if $\varphi_L^{\sf p.A}|_{\partial S} = \id_{\partial S}$,  in general $[\varphi^{\sf p.A}_L] \neq [\varphi_L] \in \MCG(S,P)$.

Let $H:S \times [0,1] \rightarrow S$ be an isotopy from $\varphi_L$ to $\varphi_L^{\sf p.A}$.
Define a homeomorphism: 
$$
\begin{array}{ccc}
A_{\partial}:\partial S \times [0,1] & \longrightarrow & \partial S \times[0,1]\\
(x,t) & \longmapsto & (H(x,t),t)
\end{array}
$$
Gluing $S \setminus P$ and annuli $\partial S\times[0,1]$ along $\partial (S \setminus P)$ and $\partial S\times \{1\}$, we obtain a surface homeomorphic to $S \setminus P$. 
We rename the original surface $(S \setminus P)_{\circ}$ and call the new surface $S \setminus P$. 
That is, 
\begin{equation}\label{eq:decomposition} 
S \setminus P = (S \setminus P)_\circ \cup (\partial S\times[0,1]).
\end{equation}
We then define a homeomorphism $\varphi_L^{{\sf p.A}\star}: S \setminus P \to S \setminus P$ by 
\begin{equation}\label{eq:pA-representative}
\varphi_L^{{\sf p.A}\star}(p)=
\begin{cases}
\varphi^{\sf p.A}_L(p)  &\mbox{ if } p \in (S \setminus P)_\circ\\
A_{\partial}(x, t) & \mbox{ if } p=(x, t) \in \partial S \times [0,1]
\end{cases}
\end{equation}
Now $\varphi_L^{{\sf p.A}\star}|_{\partial S}=\id_{\partial S}$ and $\varphi_L^{{\sf p.A}\star}$ is isotopic to $\varphi_L$ relative to the boundary. We call $\varphi_L^{{\sf p.A}\star}$ a \emph{pseudo-Anosov representative of $[\varphi_L]$ relative to the boundary}. 

In the following, we may assume that $\vphi_L: S\setminus P\to S\setminus P$ is a pseudo-Anosov representative of $[\varphi_L] \in\MCG(S\setminus P)$ relative to the boundary.
We may also assume that $\widetilde{\varphi_L}: \widetilde S \setminus \widetilde P \to \widetilde S \setminus \widetilde P$ is the lift of the modified $\varphi_L$.

Let $(\mathcal F^u, \mathcal F^s)$ be unstable/stable transverse measured foliations in $(S\setminus P)_\circ$ for $\varphi_L|_{(S \setminus P)_{\circ}}$ and $\lambda>1$ be its dilatation (see Farb-Margalit's book \cite[Part 3]{FM} for the definitions). 
The pair $(\mathcal F^u, \mathcal F^s)$ share the same singularity set. 
Away from the singularities, leaves of $\mathcal F^u$ and $\mathcal F^s$ intersect transversely. At a singularity, a number of half-lines meet which we call prongs for the singularity (cf. \cite[Figure 11.2]{FM}).

Every puncture point $p\in P$ serves as a $k$-prong-singularity of $(\mathcal F^u, \mathcal F^s)$ for some $k = k(p) \geq 1$.
When $k=2$ the singularity is just a regular point.
Collapsing a boundary component $C$ to a point yields a $k$-pronged singularity for some $k =k(C) \geq 1$. 
There are possibly non-puncture or non-boundary singularities that are interior points of $S\setminus P$.
We recall the fact (cf. \cite[page 305]{FM}):  
\begin{equation}\label{requirement} 
   \begin{array}{l} 
    \mbox{A $1$-pronged singularity only occurs at a puncture}  \\
     \mbox{or a boundary component.}  \\
   \end{array}
\end{equation}  

By the (genuine) covering map $\pi:\widetilde{S} \setminus \widetilde{P} \rightarrow S\setminus P$, the decomposition (\ref{eq:decomposition}) naturally induces a decomposition $$\widetilde{S} \setminus \widetilde{P} = (\widetilde{S} \setminus \widetilde{P})_\circ \cup (\partial \widetilde{S} \times [0,1]),$$ which is preserved by the lift $\widetilde{\varphi_L}$.
The preimage of $(\mathcal F^u, \mathcal F^s)$ under the covering map $\pi:(\widetilde{S} \setminus \widetilde{P})_\circ \rightarrow (S\setminus P)_\circ$ defines a pair of transverse measured foliations on $(\widetilde S\setminus \widetilde{P})_\circ$ denoted by $(\widetilde{\mathcal F^u}, \widetilde{\mathcal F^s})$.
By the commutativity (\ref{eq:new-commutativity}) we know that $(\widetilde{\mathcal F^u}, \widetilde{\mathcal F^s})$ are transverse measured foliations for $\wvphi|_{(\widetilde{S} \setminus \widetilde{P})_\circ}$ with the same dilatation $\lambda$. Thus, $\wvphi$ is a pseudo-Anosov representative of $[\wvphi]$ relative to the boundary. In particular, the commutativity (\ref{eq:new-commutativity}) remains hold for the pseudo-Anosov representatives $\wvphi$ and $\vphi_L$ relative to the boundary.

For the pseudo-Anosov case, the FDTC $c([\varphi_L],C)$ counts how much $\vphi_L$ twists the singular leaves (prongs) of $(\mathcal F^u, \mathcal F^s)$ near the boundary component $C$: 
Assume that there are $k=k(C)$ prongs emanating from $C=C\times \{1\}$.
The $k$ prongs meet $C$ in $k$ cyclically ordered (following the boundary orientation) and evenly distributed points $x_i$ for $i\in \mathbb Z/k\mathbb Z$.  
The map $\vphi_L$ takes the line segment $x_0\times [0,1]$ to an arc that starts at $x_0\times\{0\}$, winds along $C$ in the positive direction for $\frac{n}{k}$ times for some $n \in \Z$, and exits the annular region at $x_{n}\times \{1\}$. 
In other words, identifying $C$ with the unit circle in the complex plane, 
$\varphi_L|_{C \times [0,1]}$ is isotopic relative to the boundary to the map of $\frac{n}{k}$ twist:
$$C \times [0,1]  \ni (z,t) \longmapsto (z e^{\frac{2\pi  n t \sqrt{-1}}{k}}, t) \in C \times [0,1]$$
By \cite[p.433]{HKM} the FDTC $c([\varphi_L],C)$ is defined to be $\frac{n}{k}$.

Let $d:=d(\pi,\widetilde{C})$. Identify $\widetilde C$ with the unit circle in $\mathbb C$.
On the annular region $\widetilde C\times [0,1]$ we may assume the covering map $\pi$ is 
$$\pi|_{\widetilde C\times [0,1]}: \widetilde C\times [0,1]\ni(z, t) \mapsto (z^d, t)\in C\times[0,1]$$ so that $\pi$ takes $\widetilde C\times [0,1]$ exactly onto $C\times [0,1]$. 
The $k$ prongs in $\mathcal F^u$ (resp. $\mathcal F^s)$ emanating from $C \times \{1\}$ lift to $d(\pi,\widetilde{C})\times k$ prongs in $\widetilde{\mathcal F^u}$ (resp. $\widetilde{\mathcal F^s})$ emanating from $\widetilde C \times \{1\}$. 
By the commutativity (\ref{eq:new-commutativity}) $\widetilde{\vphi_L}|_{\widetilde C\times [0,1]}$ is isotopic to 
$$\widetilde C\times [0,1]\ni (z, t) \longmapsto  (z e^{\frac{2\pi n \sqrt{-1}}{dk}t}, t) \in\widetilde C\times [0,1],$$
which gives $c([\wvphi], \widetilde{C}) = \frac{n}{dk}$. 
That is, $c([\varphi_L],C)= \frac{n}{k}= c([\wvphi], \widetilde{C}) \cdot d(\pi,\widetilde{C})$. See Figure \ref{fig:FDTC}.
The second equality of (\ref{eq:key}) is obtained by: 
$$c(\phi,\cL,C)= c([\varphi_{L}],C)=c([\wvphi],\widetilde{C})\cdot d(\pi, \widetilde{C})= c(\widetilde{\phi},\widetilde{\cL},\widetilde{C})\cdot d(\pi, \widetilde{C}).
$$

\begin{figure}[htbp]
\begin{center}
\includegraphics[width=90mm, bb=0 0 398 297]{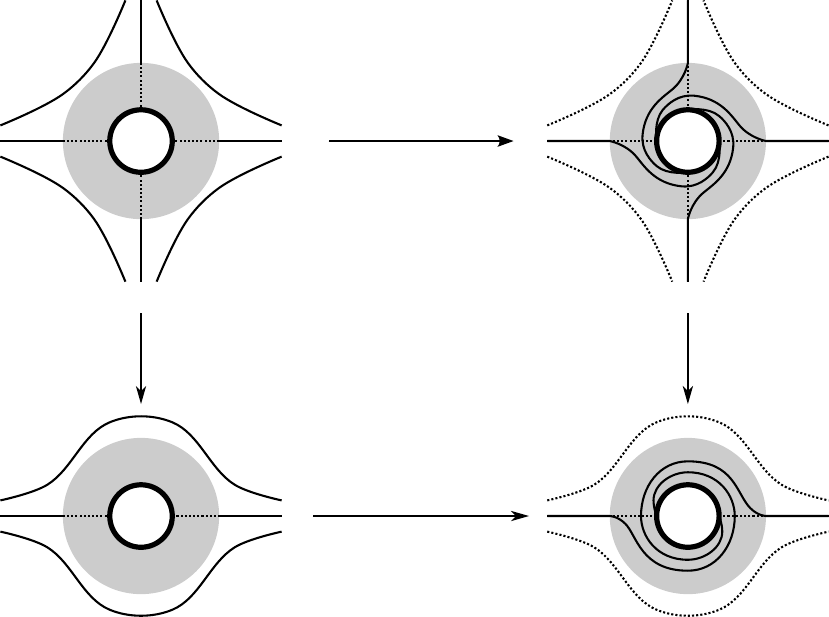}
\begin{picture}(0,0)
\put(-205,155){$\widetilde{C}$}
\put(-205,40){$C$}
\put(-135,40){$\varphi_L$}
\put(-155,155){$\widetilde{\varphi_L}$ or $f(\widetilde{\varphi_L})$}
\put(-185,10){$\mathcal{F}^u$}
\put(-185,120){$\widetilde{\mathcal{F}}^u$}
\put(-225,80){$\pi$}
\put(-45,80){$\pi$}
\end{picture}
\caption{
The shaded annuli are $C \times [0,1]$ and $\widetilde{C} \times[0,1]$ and unshaded regions are $(S\setminus P)_{\circ}$ and $(\widetilde{S}\setminus \widetilde{P})_{\circ}$.
In the shaded regions, $\varphi_L$ gives $\frac{n}{k} = c([\varphi_L],C)$ rotation, and both $\widetilde{\varphi_L}$ and $f(\widetilde{\varphi_L})$ give $\frac{n}{k d(\pi,\widetilde{C})} = c([\widetilde{\varphi_L}], \widetilde{C})$ rotation.}
\label{fig:FDTC}
\end{center}
\end{figure}

It remains to show the first equality in (\ref{eq:key}). 
To do so, we will extend $(\widetilde{\mathcal F^u}, \widetilde{\mathcal F^s})$ on $(\wS \setminus \wP)_\circ$ to $\wS_\circ$, where $\wS_\circ$ is the surface obtained from $(\wS \setminus \wP)_\circ$ by filling the punctures. 
Then we will argue that they are transverse measured foliations of the pseudo-Anosov map $f(\widetilde{\vphi_L}):\widetilde S\to\widetilde S$ with the same dilatation $\lambda$, where $f: \Homeo(\widetilde S, \widetilde P, \partial \widetilde S)\to \Homeo(\widetilde S, \partial \widetilde S)$ is the forgetful map.

Every $p\in P$ serves as a $k$-pronged singularity of $(\mF^u, \mF^s)$ for some $k=k(p) \geq 1$. 
Each pre-image point $\widetilde{p} \in \pi^{-1}(p)$ is a $kd$-pronged singularity of $(\mF^u, \mF^s)$ for some $d=d(\widetilde p)$. The fully ramified assumption implies $d>1$. Therefore each point in $\widetilde P$ is a singularity of $(\widetilde{\mathcal F^u}, \widetilde{\mathcal F^s})$ with more than one prong. 
After filling in all the the punctures in $\widetilde P$, the natural extension of the foliations to $\wS_\circ$ does not violate the rule (\ref{requirement}). 
If $\widetilde p \in \pi^{-1}(p) \subset \wP$ is a 2-prong singularity (in other words $kd=2$), then filling in the puncture yields a non-singular point. 
The extended foliations are equipped with the same transverse measures. 
We denote the new transverse measured foliations in $\wS_\circ$ by the same name $(\widetilde{\mathcal F^u}, \widetilde{\mathcal F^s})$.

Since $f(\widetilde{\vphi_L})|_{\widetilde S_\circ}:\widetilde S_\circ \to\widetilde S_\circ$ is the extension of $\widetilde{\vphi_L}|_{(S\setminus \widetilde P)_\circ}:(\widetilde S\setminus \widetilde P)_\circ \to (\widetilde S\setminus \widetilde P)_\circ$ we see the map $f(\widetilde{\vphi_L})|_{\widetilde S_\circ}$ preserves the foliations and has the same dilatation $\lambda$. 
It follows that $f(\wvphi):S\to S$ is a pseudo-Anosov representative of $\widetilde\phi={\bf f}[\wvphi] \in \MCG(\widetilde S)$ relative to the boundary. 
Since $f(\wvphi)|_{\widetilde{C} \times[0,1]} = \wvphi|_{\widetilde{C} \times[0,1]}$ we have
\[ c(\widetilde{\phi},\widetilde{C}) = c([f(\wvphi)], \widetilde C)= c([\wvphi], \widetilde C)=c(\widetilde{\phi},\widetilde{\cL}, \widetilde{C}).\]

\noindent
{\bf Case 2:} $[\varphi_L]$ is periodic.\\

\noindent
Next, suppose that $[\varphi_L]$ is periodic. 
There exists $N\in\mathbb Z$ such that $\varphi_L^N$ is freely isotopic to $id_{(S, P)}$. 
This means that there exists $M\in\mathbb Z$ and 
$T\in\MCG(S, P)$ which is a product of Dehn twists about components of $\partial S\setminus C$ such that
\begin{equation}\label{eq:downstair}
[\varphi_L]^N = (T_C)^M  T \in \MCG(S, P).
\end{equation}
This gives 
\begin{equation}\label{eq:downstair2}
c(\phi, \cL, C) =c([\varphi_L], C)=\frac{M}{N}.
\end{equation}

Suppose that the branched covering $\pi:(\widetilde S, \widetilde P)\to(S, P)$ is $\delta$-fold.
Let $D=\delta!$. Then for every boundary component $X$ of $\widetilde{S}$, $d(\pi,X)$ divides $D$. In particular, $d:=d(\pi, \widetilde C)$ divides $D$. 
Taking the $D$th power of (\ref{eq:downstair}) we get 
\[ 
[\varphi_L]^{ND} = (T_C)^{MD} \ T^D \in \MCG(S, P).
\]
Upstairs in the covering, this yields
\begin{equation}\label{eq:upstair-puncture}
[\wvphi]^{ND} = (T_{\widetilde C})^{M D/d} \ T'\in \MCG(\widetilde S, \widetilde P)
\end{equation}
for some product $T'$ of Dehn twists about components of $\partial \widetilde S\setminus \widetilde C$. 
Since $c(T',\widetilde C)=0$ this gives 
\begin{equation}\label{eq:upstair-puncture2}
c(\widetilde\phi, \widetilde\cL, \widetilde C) = c([\wvphi],{\widetilde C}) = \frac{M}{dN}.
\end{equation}
Filling the puncture points of $\widetilde P$ from (\ref{eq:upstair-puncture}) we obtain 
\begin{equation}\label{eq:upstair}
\wphi^{ND}=({\bf f}[\widetilde\varphi])^{ND} = ({\bf f}(T_{\widetilde C}))^{M D/d}{\bf f}(T') \ \ \ \mbox{ in } \MCG(\widetilde S).
\end{equation}
where ${\bf f}:\MCG(\widetilde{S},\widetilde{P}) \rightarrow \MCG(\widetilde{S})$ denotes the forgetful map. 
Since $\chi(\widetilde S)<0$ the space $\widetilde S$ is not an annulus. Thus, any two distinct boundary components $\widetilde C_1, \widetilde C_2$ of $\widetilde S$ do not bound an annulus and we have
\begin{equation}
\label{eqn:observation} 
{\bf f}(T_{\widetilde C_1}) \neq {\bf f}(T_{\widetilde C_2}) \in \MCG(\widetilde{S})
\end{equation} (cf.~Remark~\ref{rem:0} below).
Since $T'$ is a product of Dehn twists about components of $\partial \widetilde S\setminus \widetilde C$, by the observation (\ref{eqn:observation}) we have $c({\bf f}(T'),\widetilde{C})= 0$.
Knowing that ${\bf f}(T_{\widetilde C})$ is a Dehn twist along $\widetilde{C}$, 
equation (\ref{eq:upstair}) gives 
\begin{equation}
\label{eq:upstair2}
c(\widetilde\phi, \widetilde C) = c({\bf f}[\widetilde{\varphi}],\widetilde C)=\frac{M}{dN}.
\end{equation}
Equations (\ref{eq:downstair2}), (\ref{eq:upstair-puncture2}) and (\ref{eq:upstair2}) give the desired (\ref{eq:key}). \\

\noindent
{\bf Case 3:} $[\varphi_L]$ is reducible.\\

\noindent 

First we review the definition of the FDTC for reducible mapping classes.  
Although this formulation perhaps has been known, we give a detailed exposition here since we were unable to find suitable references. 

Let $\mathcal{C}=\{c_1,\ldots, c_l\}$ be an {\em adequate reduction system} of the reducible mapping class $[\varphi_L]$ that decompose $\varphi_L$ into periodic and pseudo-Anosov pieces \cite{BLM}. Namely, 
\begin{itemize}
\item[($\star$)] The map $\varphi_L$ fixes the union of simple closed curves $c_1\cup \cdots \cup c_l$, and the restriction of $\varphi_L$ to each connected component of $S \setminus (c_1\cup \cdots \cup c_l)$ is either periodic or pseudo-Anosov.
\end{itemize}

Let $\nu(\mathcal{C}) := \bigcup_{c \in \mathcal{C}} \nu(c)$ be an open collar 
neighborhood of $\mathcal{C}$. With no loss of generality, we may assume that $\vphi_L(\nu(\mathcal{C})) = \nu(\mathcal{C})$. 
Let $S_1,\ldots, S_m$ be the connected components of $S \setminus \nu(\mathcal{C})$ and $P_i:=S_i\cap P$.

Take $n>0$ so that $\vphi_L^n(c)=c$ for each curve $c \in \mathcal{C}$.
Note that $n$ depends on the choice of $\mathcal C$ and no need to be the smallest such number. 
We take a diffeomorphism
\[ \varrho_L : (S,P) \rightarrow (S,P) \]
isotopic to $\vphi_L^n$ relative to $\partial S$ so that 
\begin{itemize}
\item[(i)] $\varrho_L$ fixes $\partial \nu(\mathcal{C})$ pointwise.
\end{itemize}
Then $\varrho_L$ fixes $S_i$ set-wise for each $i=1,\dots,m$, and its restriction $\varrho_L|_{S_i}$ fixes $\partial S_i$ pointwise.
Therefore we can consider its mapping class $[\varrho_L|_{S_i}] \in \MCG(S_i,P_i)$. With a further isotopy, we may assume that 
\begin{itemize}
\item[(ii)] If $[\varrho_L|_{S_i}]$ is of pseudo-Anosov, $\varrho_L|_{S_i}$ is a pseudo-Anosov representative relative to the boundaries (in the sense of (\ref{eq:pA-representative})). 
\end{itemize}
If (i) and (ii) are satisfied, we call $\varrho_L$ a \emph{good representative} of the reducible mapping class $[\vphi_L^n]$ associated to $\mathcal C$.
A good representative is only defined up to Dehn twists along curves in $\partial \nu(\mathcal C)$, but this ambiguity causes no trouble in the following argument.

Using the good representative, the FDTC of the reducible map is defined as follows.
Let $S_C$ be the connected component of $S \setminus \nu(\mathcal{C})$ that contains the boundary component $C$, and let $P_C:=P \cap S_C$ be the set of punctures in $S_C$.
We define the FDTC of $[\varphi_L]$ by 
\begin{equation}\label{eq:reducible FDTC1}
c([\varphi_L],C) := \tfrac{1}{n}c([\varrho_{L}|_{S_{C}}],C).
\end{equation}

\begin{claim}
The definition (\ref{eq:reducible FDTC1}) is well-defined; that is, it is independent of choices of good representative $\varrho_L$, $n$, and adequate reduction system $\mathcal C$.
\end{claim}

\begin{proof}
Suppose that $\varrho_L|_{S_{C}}$ is pseudo-Anosov. \\

\noindent
(1) Suppose that $\varrho_L$ and $\varsigma_L$ are good representatives of $[\varphi_L^n]$ associated to $\mathcal C$.
Then  $\varrho_L|_{S_C}$ and $\varsigma_L|_{S_C}$  have the same stable/unstable foliations, so $c([\varrho_L|_{S_C}],C)=c([\varsigma_L|_{S_C}], C)$. \\

\noindent
(2) Next suppose that for $n, n'>0$, $\varphi_L^n(c)=\varphi_L^{n'}(c)=c$ for each $c\in\mathcal C$. Suppose that $\varrho_L$ (resp. $\varrho'_L$) is a good representative of  $[\varphi_L^n]$ (resp. $[\varphi_L^{n'}]$) associated to $\mathcal C$. 
Then both $(\varrho_L)^{n'}$ and $(\varrho'_L)^n$ are good representatives of $[\varphi_L^{n n'}]$ associated to $\mathcal C$. 
The discussion (1) gives that their FDTCs are equal at $C$. Therefore, 
$$
\tfrac{1}{n} c ( [\varrho_L|_{S_C}], C)= 
\tfrac{1}{nn'} c([(\varrho_L)^{n'}|_{S_C}], C)=
\tfrac{1}{nn'} c([(\varrho'_L)^{n}|_{S_C}], C)=
\tfrac{1}{n'} c([\varrho'_L |_{S_C} ], C).
$$
\noindent
(3) Let $\mathcal C^{\sf can}$ be the {\em canonical} reduction system of $[\varphi_L]$, the smallest adequate reduction system \cite[13.2.2]{FM}({\em essential} reduction system in \cite{BLM}). 
Let $S_C^{\sf can}$ be the connected component of $S \setminus \nu(\mathcal{C}^{\sf can})$ that contains the boundary component $C$. 
Since $\mathcal{C}^{\sf can} \subset \mathcal{C}$, $S_{C}$ is a subsurface of $S_{C}^{\sf can}$ and pseudo-Anosov representatives  $\varrho_L|_{S_C}$ and  $\varrho_L|_{S_C^{\sf can}}$ have the same foliations, so $S_C = S_C^{\sf can}$.
Thus, (\ref{eq:reducible FDTC1}) is independent of a choice of reduction system $\mathcal C$. \\

\medskip

Next, we assume that $\varrho_L|_{S_C}$ is periodic. \\

\noindent
(1$'$) Suppose that $\varrho_L$ and $\varsigma_L$ are good representatives of $[\varphi_L^n]$ associated to $\mathcal C$. 
Since both $\varrho_L|_{S_C}$ and $\varsigma_L|_{S_C}$ are of periodic type, there exist $N\geq 1$ and $A, B \in \Z$ such that 
\[
(\varrho_L|_{S_C})^N = (T_{C})^A \cdot (\mbox{Dehn twists along other boundary components of }S_C), 
\]
\[
(\varsigma_L|_{S_C})^N = (T_{C})^B \cdot (\mbox{Dehn twists along other boundary components of }S_C). 
\]
Take an essential arc $\alpha \subset S_C$ that starts and ends at $C$. 
Then by
$$
(T_C)^A(\alpha)=(\varrho_L)^N(\alpha)=(\varphi_L)^{nN}(\alpha)=(\varsigma_L)^N(\alpha)=(T_C)^B(\alpha),
$$
we have $A=B$, so $c([\varrho_L|_{S_C}],C)=c([\varsigma_L|_{S_C}], C)$. \\

\noindent 
(2$'$)
The above argument (2) of the independency of choice of $n$ can be recycled here. \\

\noindent
(3$'$)
Take the canonical reduction system $\mathcal{C}^{\sf can}$  of the reducible class $[\varphi_L]$. \\
Then $\varrho_L|_{S_C^{\sf can}}$ is also periodic. 
Since $S_{C}\subset S_{C}^{\sf can}$, $\varrho_L|_{S_C}$ is the restriction of $\varrho_L|_{S_C^{\sf can}}$ to $S_C$ and there exist $M\geq 1$ and $A, B \in \Z$ such that 
\begin{eqnarray*}
(\varrho_L|_{S_C^{\sf can}})^M &=& (T_{C})^{A} \cdot (\mbox{Dehn twists along other boundary components of }S_C^{\sf can} ), \\
(\varrho_L|_{S_C})^M &=& (T_{C})^{B} \cdot (\mbox{Dehn twists along other boundary components of }S_C).
\end{eqnarray*}
By the same argument as in (3) using an arc $\alpha\subset S_C$ we obtain $A=B$. 
Thus,
\[  c(\varrho_L|_{S_C},C) = \tfrac{A}{M}= c(\varrho_L|_{S_C^{\sf can}}, C). \] 
Since  $\mathcal{C}^{\sf can}$ is unique for $[\varphi_L]$, the definition (\ref{eq:reducible FDTC1}) is independent of a choice of $\mathcal C$.
\end{proof}

Having defined the FDTC for the reducible case, we turn to the proof of the assertion.

\begin{claim}\label{claim-preimage of c}
For each curve $c \in \mathcal{C}$, a connected component $\widetilde{c}$ of the preimage $\pi^{-1}(c)$ remains essential in $\widetilde{S}$ (after filling in the punctures).
\end{claim}

\begin{proof}
If $\widetilde{c}$ is inessential in $\widetilde{S}$, then either
\begin{itemize}
\item[(a)] $\widetilde{c}$ bounds a disk $D \subset \widetilde S$ with branch points, or 
\item[(b)] $\widetilde{c}$ and a boundary component, $C'$, of $\widetilde{S}$ cobound an annulus $A$ with branch points. 
\end{itemize}
We note that (a) cannot happen because an Euler characteristic argument implies that the restriction $\pi|_{D}$ cannot be a fully ramified branched covering. 
For (b), since the restriction $\pi|_A$ gives a fully ramified branched covering, an Euler characteristic argument implies that (b) occurs only if $\pi(A)=D^{2}$ and $\pi|_A:A \rightarrow D^2$ is a double branched covering with two branch points. 
This means that $\pi^{-1}(\partial D^{2})= \widetilde {c} \cup C'$.
Since $C'$ is a boundary component of $\widetilde S$, so is $\widetilde c$, which is a contradiction. 
\end{proof}

Let 
$$\widetilde{\varrho_{L}} :
(\widetilde{S}, \widetilde{P}) \rightarrow 
(\widetilde{S}, \widetilde{P})
$$
be a lift of the good representative $\varrho_{L}$ of $[\varphi_L^n]$.
Since $\wvphi$ preserves $\pi^{-1}(\mathcal{C})$, Claim~\ref{claim-preimage of c} shows that both $[\wvphi] \in \MCG(\widetilde S, \widetilde P)$ and $\widetilde \phi=\f[\wvphi] \in \MCG(\widetilde S)$ are reducible. Moreover, the reduction system $\pi^{-1}(\mathcal C)$ of $[\wvphi]$ 
satisfies the property $(\star)$, and $\widetilde{\varrho_L}$ is a good representative of $[\widetilde{\vphi_L}^n]$.

\medskip

Let $\widetilde{C}$ be a connected component of the preimage $\pi^{-1}
(C)$. 

Define $\wS_{\widetilde C}$ to be the connected component of the preimage of $S_C$ that contains $\widetilde C$. 
Let $\widetilde P_{\widetilde C}$ be the preimage of $P_C$ that are in $\wS_{\widetilde C}$.
The restriction 
$\pi|_{{\widetilde{S}}_{\widetilde C}} : (\widetilde{S}_{\widetilde C}, \widetilde{P}_{\widetilde C}) \to (S_C, P_C)$
is a fully ramified branched covering. 
Since $S_C$ has more than one boundary component and is not an annulus, $\chi(\wS_{\widetilde C}) < 0$.

Since $[\varrho_L|_{S_C}] \in \MCG(S_C,P_C)$ is either pseudo-Anosov or periodic, 
by the arguments in Case 1 and Case 2, and the definition of the FDTC for reducible case (\ref{eq:reducible FDTC1}), it follows that
$$ 
c(\widetilde{\phi},\widetilde{\cL}, \widetilde{C}) = 
\tfrac{1}{n}c([\widetilde{\varrho_{L}}|_{\widetilde{S}_{\widetilde{C}}}], \widetilde C)=
\tfrac{1}{n}c([\varrho_{L}|_{S_C}], C)/d(\pi,\widetilde{C})= c(\phi,\cL,C)/d(\pi,\widetilde{C})
$$
and 
$$ 
c(\widetilde{\phi}, \widetilde{C}) = 
\tfrac{1}{n} c(\f[\widetilde{\varrho_{L}}|_{\widetilde{S}_{\widetilde{C}}}], \widetilde C)= 
\tfrac{1}{n} c([\widetilde{\varrho_{L}}|_{\widetilde{S}_{\widetilde{C}}}], \widetilde C)= 
c(\widetilde{\phi},\widetilde{\cL}, \widetilde{C}).
$$

\end{proof}

\begin{remark}\label{rem:0}
If $\chi(\widetilde{S})=0$ (i.e., $\widetilde{S}$ is an annulus) an Euler characteristic argument shows that $S$ is a disk, $|P|=2$, and  the degree of the branched covering $\pi$ at each branch point is $2$. 
Theorem~\ref{theorem:covering} does not hold but we have 
\begin{equation}\label{chi0case}
c(\wphi, \widetilde C)=2 c(\widetilde{\phi},\widetilde{\cL}, \widetilde{C}) =2 c(\phi,\cL,C)/d(\pi, \widetilde{C}).
\end{equation}

Let $D$ be a disk and $A$ be an annulus. 
Consider a double branched covering $\pi: A \rightarrow D$ branched at two points $P=\{p_1,p_2\} \subset D$. 
Denote $\widetilde P=\pi^{-1}(P)=\{\widetilde{p_1}, \widetilde{p_2}\} \subset A$. 
Let $\phi=\id \in \MCG(D)$ and $\cL$ be the braid closure of a 2-stranded braid with positive half-twist switching $p_1$ and $p_2$ counterclockwise. Thus the topological type of $\cL$ is the unknot. 

We observe that the distinguished monodromy of $\cL$ is the Artin generator $\sigma_1 \in B_2=\MCG(D^2, P)$, and the distinguished monodromy $[\wvphi] \in \MCG(A, \widetilde P)$ of the lift $\widetilde{\cL}$ 
permutes $\widetilde P$ non-trivially as $\wvphi(\widetilde{p_i})=\widetilde{p_{i+1}}$. 
Forgetting the punctures $\widetilde P$ we note that $\wphi=\f[\wvphi]\in\MCG(A)$ is the positive Dehn twist about the core of the annulus. 

Let $C=\partial D$ and $\widetilde C$ be one of the boundary components of the annulus $A$. 
The above observation implies 
$c(\wphi,\widetilde{C})=1$, 
$c(\widetilde{\phi},\widetilde{\cL}, \widetilde{C})=\frac{1}{2}$,  
$c(\id, \cL, C) = \frac{1}{2}$ and $d(\pi, \widetilde{C}) = 1.$
We confirm (\ref{chi0case}) is satisfied.
\end{remark}

The following statement is proven in the proof of Theorem~\ref{theorem:covering}. It says Nielsen-Thurston types are preserved under fully ramified branched covers. 

\begin{corollary}\label{cor:NT-type-preserved}
Assume the same setting as Theorem~\ref{theorem:covering}. 
If the distinguished monodromy $[\vphi_L]$ is pseudo-Anosov (resp. periodic and reducible) then both $[\wvphi]$ and $\wphi= \f[\wvphi]$ are pseudo-Anosov (resp. periodic and reducible). 
\end{corollary}

\section{Rightveering property and branched covering}\label{section:right-veering}

Our goal of the section is to show Theorem~\ref{theorem:reducible-right-veering} that gives a connection between right-veeringness and the FDTC for reducible case (the pseudo-Anosov and periodic cases were established by Honda, Kazez and Mati\'c \cite{HKM}), then to show Corollary~\ref{cor:covering-right-veering} that  the right-veering property is preserved under fully ramified branched covers.

We start with reviewing the definition of right-veering. 
We say that two properly embedded arcs $\gamma$ and $\gamma'$ are in \emph{tight position} if $\gamma = \gamma'$ or $\gamma$ and $\gamma'$ attain the minimum geometric intersection.

\begin{definition}
\label{definition:rv-order}
Let $\gamma$ and $\gamma'$ be properly embedded arcs in $S \setminus P$ sharing the same base point $\ast$ and in tight position.
We denote $\gamma \prec^{S \setminus P}_{\sf right} \gamma'$ if $\gamma'$ lies on the right side of $\gamma$ in a small neighborhood of $\ast$.
We denote $\gamma \preceq^{S \setminus P}_{\sf right} \gamma'$ if $\gamma=\gamma'$ or $\gamma \prec^{S \setminus P}_{\sf right} \gamma'$.  
\end{definition}

We use the same symbol $\prec^{S \setminus P}_{\sf right}$ for \emph{isotopy classes} of arcs sharing the same base point, and denote by $[\gamma] \prec^{S \setminus P}_{\sf right} [\gamma']$ if $\gamma \prec^{S \setminus P}_{\sf right} \gamma'$ for some representatives $\gamma$ and $\gamma'$.
We note that $\prec^{S \setminus P}_{\sf right}$ is a total ordering on the set of isotopy classes of arcs. Total orderings for other surfaces can be defined similarly (eg. $\prec^{\widetilde{S}}_{\sf right}$). 

\begin{definition}
A mapping class $\psi \in \MCG(S,P)$ is \emph{right-veering} with respect to a boundary component $C$ of $S$ if $[\gamma] \preceq^{S \setminus P}_{\sf right} \psi([\gamma])$ for every properly embedded arc $\gamma$ based on $C$.
\end{definition}
It is immediate that right-veeringness is preserved under taking powers: 
\begin{claim}\label{claim-RV-preserved}
$\psi$ is right-veering with respect to $C$ if and only if some positive power of $\psi$ is right-veering with respect to $C$.
\end{claim}

Here is Honda, Kazez and Mati\'c's right-veering criterion  in terms of  the FDTC.
\begin{theorem}[Right-veering criterion]\cite[Section 3]{HKM}
\label{theorem:HKM-RV-criterion}
\begin{enumerate}
\item Let $\psi$ be pseudo-Anosov. $\psi$ is right-veering with respect to $C$ if and only if the FDTC $c(\psi, C)>0$. 
\item Let $\psi$ be periodic. $\psi$ is right-veering with respect to $C$ if and only if
\begin{enumerate}
\item the FDTC $c(\psi, C)>0$, or,
\item the FDTC $c(\psi, C)=0$ and $c(\psi,C')\geq 0$ for every boundary components $C'\neq C$.
\end{enumerate}
\end{enumerate}
\end{theorem}

We extend the right-veering criterion to the reducible case in Theorem \ref{theorem:reducible-right-veering}.
To this end, we introduce an \emph{interior fractional Dehn twist coefficient}:

Let $\psi \in \MCG(S,P)$ be reducible with an adequate reduction system $\mathcal C$. 
Let $S_1,\ldots, S_m$ be the connected components of $S \setminus \nu(\mathcal{C})$. 
As in the proof of Theorem~\ref{theorem:covering} (Case 3), take $n>0$ so that $\psi^{n}$ fixes each curve $c \in\mathcal{C}$, and take a good representative $\varrho$ of $\psi^n$ associated to $\mathcal C$.
Choose a diffeomorphism $\varrho_i:S \to S$ supported on $S_i$ so that $\varrho_i|_{S\setminus S_i} = \id_{S\setminus S_i}$ and $\psi^n$ decomposes as
\begin{equation}
\label{eqn:decomposition}
\psi^n =[\varrho]=[\varrho_1][\varrho_2]\cdots[\varrho_{m}]\cdot \prod_{c \in \mathcal{C}} {T_{c}}^{n_c} 
\end{equation}
for some $n_c \in \Z$. Retaking the good representative $\varrho$, we may assume that $\varrho|_{S_i}=\varrho_i|_{S_i}$.

For $c \in \mathcal{C}$ we denote by $C_{\pm}$ the boundary curve $c\times \{\pm 1\}$ of  the closure of the collar neighborhood $\overline{\nu(c)} \cong c\times[-1,1]$, and denote by $S_{c+}$ (resp. $S_{c-}$) (possibly equal) be the connected component of $S \setminus \nu(\mathcal C)$ that contains $C_+$ (resp. $C_-$). 

\begin{definition}
We define the \emph{interior fractional Dehn twist coefficient} of the reducible mapping class $\psi$ at the curve $c \in \mathcal C$ as
\begin{equation}\label{def of interior}
\omega(\psi,c) := \frac{1}{n}\left(c([\varrho|_{S_{c+}}], C_+) +c([\varrho|_{S_{c-}}], C_-) + n_c \right). 
\end{equation}
The value $\omega(\psi,c)$ is independent of choices of $\varrho$, $n$, and $\mathcal C$.  
\end{definition}

Since the interior FDTC comes from the FDTC of $\varrho_i|_{S_i}$ which is either pseudo-Anosov or periodic, it shares the same properties as the usual FDTC including the homogeneity:
\begin{equation}\label{eq:homogenious}
\omega(\psi^{m},c)=m\ \omega(\psi,c) \mbox{ for all } m \in \Z.
\end{equation}

So far, we have assumed that $\mathcal C$ is just an adequate reduction system of $\psi$. 
If we assume $\mathcal C$ to be canonical, $w(\psi, c)$ has the following nice property.

\begin{lemma}
\label{lemma:intFDTC}
If $\mathcal{C}$ is canonical and both $\varrho|_{S_{c+}}$ and $\varrho|_{S_{c-}}$ are periodic, then $\omega(\psi,c) \neq 0$. 
\end{lemma}
\begin{proof}
Assume that both $\varrho|_{S_{c+}}$ and $\varrho|_{S_{c-}}$ are periodic.
Take $k>0$ so that for some $m_{\pm}\in \Z$
\begin{align*}
[\varrho|_{S_{c+}}]^{k} = T_{C_+}^{m_{+}}\cdot(\mbox{Dehn twists along other boundary components of } S_{c+}),\\
[\varrho|_{S_{c-}}]^{k} = T_{C_-}^{m_{-}}\cdot(\mbox{Dehn twists along other boundary components of } S_{c-}).
\end{align*}
Let $S' = S_{c+} \cup \nu(c)\cup S_{c_-}$. Since $C_+$, $c$, $C_-$ are isotopic in $S'$, if $\omega(\psi,c)=0$ we have
\begin{align*}
[\varrho|_{S'}^{k}] &= T_{C_+}^{m_{+}}\cdot(\mbox{Dehn twists along other boundary components of } S_{c+})\cdot T_{c}^{n_c}\\
& \qquad \cdot T_{C_-}^{m_{-}}\cdot(\mbox{Dehn twists along other boundary components of } S_{c-})\\
&= (\mbox{Dehn twists along boundary components of } S').
\end{align*}
This means that, $\mathcal{C}\setminus\{c\}$ is also an adequate reduction system of $\psi$. 
This contradicts that $\mathcal{C}$ is canonical. 
\end{proof}

We say that a boundary component $C$ of $S_i$ is \emph{exterior boundary} if $C \subset \partial S$, and \emph{interior boundary} associated to a curve in $\mathcal C$ otherwise.

\begin{theorem}
\label{theorem:reducible-right-veering}
Let $\psi \in \MCG(S,P)$ be reducible with the canonical reduction system $\mathcal C^{\sf can}$.
Let $\varrho$ be a good representative of $\psi^n$ associated to $\mathcal C^{\sf can}$.
Then $\psi$ is right-veering with respect to $C \subset \partial S$ if and only if either
\begin{enumerate}
\item $c(\psi,C)>0$, or,
\item $c(\psi,C)=0$, $[\varrho|_{S_C}]$ is periodic, and for every boundary component $C' \neq C$ of $S_C$ we have: 
\begin{enumerate}
\item If $C'$ is an exterior boundary, then $c(\psi,C')\geq 0$. 
\item If $C'$ is an interior boundary associated to $c \in \mathcal{C}^{\sf can}$, then $\omega(\psi,c)>0$.
\end{enumerate}
\end{enumerate}
\end{theorem}

Note that Theorem~\ref{theorem:reducible-right-veering} is almost the same as Theorem \ref{theorem:HKM-RV-criterion} except that we require strict positivity of the interior FDTC in (2b).

\begin{proof}[Proof of Theorem~\ref{theorem:reducible-right-veering}]

Since the sign of the FDTC, the sign of the interior FDTC, and right-veeringness are preserved under taking powers of the map (see (\ref{eq:homogenious}) and Claim~\ref{claim-RV-preserved}), we may assume that $\psi=[\varrho]$.
By taking even more power of $\psi$ if needed, we may further assume: 
\begin{itemize}
\item[(i)] $c(\varrho|_{S_i},C') \in \Z$ for every exterior/interior boundary component $C'$ of $S_i$.
\item[(ii)] If $\varrho|_{S_i}$ is periodic then it is a product of boundary Dehn twists. 
\item[(iii)] 
If $C'$ is an interior boundary component of $S_C$ associated to $c\in\mathcal C^{\sf can}$, we assume that $S_{c+} =S_C$.
\item[(iv)] For all $c\in\mathcal C^{\sf can}$, $w(\psi,c)=c([\varrho|_{S_{c+}}], C_+)$ and $c([\varrho|_{S_{c-}}], C_-) = n_C=0$ in (\ref{def of interior}).
\end{itemize}

Let $\gamma$ be a properly embedded arc in $S \setminus P$ based on $C$ and minimally intersecting $\mathcal{C}^{\sf can}$. 
Let $\gamma_0$ be the properly embedded arc in $S_{C} \setminus P_C$ that is the connected component of $\gamma \cap S_C$ containing the base point.
With an isotopy relative to the endpoints, we may view $\gamma$ as 
\[ \gamma= \gamma_0 \ast \gamma_1 \ast \cdots \ast \gamma_{2k},\]
where 
$\gamma_{2i}$ is a connected component of $\gamma \cap S_j$ for some $j$ and properly embedded in $S_j\setminus P_j$,  
and 
$\gamma_{2i+1}$  is a connected component of $\gamma \cap \nu(c)$ for some $c\in\mathcal C^{\sf can}$ such that $\gamma_{2i+1}= \{p\}\times (-1,1)$ under a homeomorphism $\nu(c)\simeq S^1 \times (-1, 1)$.  

The image 
\[ 
\varrho(\gamma) = \varrho(\gamma_0) \ast \varrho(\gamma_1) \ast \varrho(\gamma_2) \ast \cdots \ast \varrho(\gamma_{2k})
\]
and $\gamma$ may not be in tight position. 
If $\varrho(\gamma_{2i})$ and $\gamma_{2j}$ form a bigon, remove the bigon by modifying $\varrho(\gamma_{2i})$ by an isotopy away from $\nu(\mathcal C)$. 
Noticing that $\varrho(\gamma_{2i+1})=\gamma_{2i+1}$ by (iv), we modify $\varrho(\gamma)$ in neighborhood of $\nu(\mathcal C)$ so that $\gamma_{2i+1} \subset \varrho(\gamma)$ and $\gamma_{2i+1}\subset\gamma$ intersect at most once and transversely. 
This can be done without creating any new bigons.
Denote by $\gamma'$ the resulting curve representing $[\varrho(\gamma)]$. 
We consider a similar decomposition 
\[ \gamma' = \gamma'_0 \ast \gamma'_1 \ast\cdots\ast \gamma'_{2k}.\]

By construction, $\gamma$ and $\gamma'$ are in tight position, and 
\begin{equation}
\label{eqn:order-preserving}
\mbox{if } \ 
\gamma \preceq_{\sf right}^{S\setminus P} \gamma' \ \mbox{ then } \ \gamma_0 \preceq_{\sf right}^{S_C\setminus P_C} \gamma'_0.
\end{equation}

\noindent
\underline{\bf Right-veering $\Rightarrow$ (1) or (2)}\\

If $\psi=[\varrho]$ is right-veering at $C$, then by \eqref{eqn:order-preserving} $[\varrho|_{S_C}]$ is right-veering at $C$, since every properly embedded arc $\gamma_C$ in $S_{C}$ based on $C$ can be extended to a properly embedded arc $\gamma$ in $S$ so that $\gamma_0=\gamma_C$. 

Theorem \ref{theorem:HKM-RV-criterion} implies that 
\begin{itemize}
\item[(I)] $c([\varrho|_{S_C}], C)>0$ (i.e., (1) holds), or 
\item[(II)] $c([\varrho|_{S_C}], C)=0$, $[\varrho|_{S_C}]$ is periodic, and $c([\varrho|_{S_C}], C')\geq0$ for all other boundary components $C'$ of $S_C$.
\end{itemize}
Therefore, it is enough to show that in Case (II) $c([\varrho|_{S_C}], C')>0$ for every interior boundary component $C'$ of $S_C$. 

Assume to the contrary that for some interior boundary component $C'$ of $S_C$ associated to some $c\in\mathcal C^{\sf can}$ we have $\omega(\psi,c)\stackrel{\tiny\textrm{(iv)}}{=}c([\varrho|_{S_C}], C')=0$.
This means that $[\varrho|_{S_{c-}}]$ is pseudo-Anosov by Lemma~\ref{lemma:intFDTC}. 

Since $c([\varrho|_{S_{c-}}], C_-) =0$ by (iv), Theorem \ref{theorem:HKM-RV-criterion} implies that $[\varrho|_{S_{c-}}]$ is not right-veering at $C_-$. 
That is, there is a properly embedded arc $\gamma_2$ in $S_{c-}\setminus P_{c-}$ based on $C_-$ such that 
\begin{equation}\label{eq:gamma2left}
[(\varrho|_{S_{c-}})(\gamma_2)] \prec_{\sf right}^{S_{c-}\setminus P_{c-}} [\gamma_2].
\end{equation}
Take a properly embedded arc $\gamma\subset S \setminus P$ based on $C$ so that
\[\gamma = \gamma_0 \ast \gamma_1 \ast \gamma_2 \ast\cdots\ast \gamma_{2k}, \]
where $\gamma_0$ is a properly embedded arc in $S_C$ joining $C$ and  $C'=C_+$, and $\gamma_1$ is a properly embedded arc in $\nu(c)$ connecting the endpoint of $\gamma_0$ and the base point of $\gamma_2$.

Since $c([\varrho|_{S_C}], C)=\omega(\psi,c)=c([\varrho|_{S_C}],C_+)=n_c=0$, by (ii) we have $\varrho|_{S_C}(\gamma_0)=\gamma_0$ and $\varrho|_{\nu(c)}(\gamma_1)=\gamma_1$.
By (\ref{eq:gamma2left}), it follows that 
\[[\varrho(\gamma)]= [\gamma']  \prec_{\sf right}^{S\setminus P} [\gamma]. \]
This contradicts the assumption that $\psi=[\varrho]$ is right-veering.
Therefore, (2) holds.\\

\noindent
\underline{\bf (1) or (2) $\Rightarrow$ right-veering}\\

By (i)--(iv) and Theorem \ref{theorem:HKM-RV-criterion}, (1) or (2) implies that $[\varrho|_{S_C}]$ is right-veering at $C$.
Thus, for any properly embedded arc in $\gamma=\gamma_0 \ast \gamma_1 \ast \gamma_2 \ast\cdots\ast \gamma_{2k}\subset S\setminus P$ based on $C$ we have 
$$[\gamma_0] \prec_{\sf right}^{S_C\setminus P_C} [(\varrho|_{S_C})(\gamma_0)] \ \ 
\mbox{ or } \ \ [\gamma_0] = [(\varrho|_{S_C})(\gamma_0)].$$

The first case gives $\gamma_0 \prec_{\sf right}^{S_C\setminus P_C} \gamma'_0$, which implies $\gamma \prec_{\sf right}^{S_C\setminus P_C} \gamma'$ since $\gamma$ and $\gamma'$ are in tight position.

The second case gives $\gamma_0=\gamma'_0$.
This happens only if $\varrho|_{S_{C}} =\varrho|_{S_{c+}}$ is periodic.
Since $c([\varrho|_{S_C}],C_+) =\omega(\psi,c)>0$ for every interior boundary $C_+$ of $S_C$ associated to $c$, $\gamma_0$ cannot end on the interior boundary $C_+$.
If $C_-$ is also interior boundary of $S_C$, by switching the roles of $C_+$ and $C_-$ we see that $\gamma_0$ cannot end on the interior boundary $C_-$, either. Thus $\gamma_0$ ends on an exterior boundary. This means that $\gamma = \gamma_0 = \gamma'_0=\gamma'$.
\end{proof}

Having established the characterization of right-veeringness in terms of the (interior) FDTC, we finally show the following.

\begin{corollary}\label{cor:covering-right-veering}
Let $\pi: (\widetilde S, \widetilde P)\to (S, P)$ be a fully ramified branched covering and $(\wS, \wphi)$ be a branched covering of $(S, \phi)$ along $\cL$. The following are equivalent:
\begin{enumerate}
\item $\cL$ is right-veering with respect to $C.$ 
\item $\widetilde{\phi}$ is right-veering with respect to $\widetilde C.$
\item The lift $\widetilde{\cL}$ is right-veering with respect to $\widetilde C.$
\end{enumerate}
\end{corollary}

\begin{proof}
When $\chi(\widetilde{S})=0$ ($\wS$ is an annulus), 
the equation (\ref{chi0case}) in Remark~\ref{rem:0} shows that right-veering-ness is carried over to $\wphi$ and $\widetilde \cL$.

Suppose next that $\chi(\widetilde{S})<0$.
In this case the equivalence of (1), (2) and (3) follows from Theorem~\ref{theorem:covering}, Theorem~\ref{theorem:HKM-RV-criterion} the right-veering criterion for periodic and pseudo-Anosov cases, and Theorem \ref{theorem:reducible-right-veering} for reducible case. 

We remark that since the interior FDTC is defined in terms of the FDTCs, Theorem~\ref{theorem:covering} shows that non-negativity/strict positivity of the interior FDTC is preserved under a fully ramified branched covering.
\end{proof}

\section{Applications}
\label{section:applications}

The FDTC has various applications to 3-manifolds and contact 3-manifolds. In this section we use Theorem \ref{theorem:covering} to study properties of fully ramified branched coverings.

\subsection{Applications to topology}

For a closed braid $\cL=[((S,\varphi),L)]$ in the open book $(S,\phi)$, the topological type of the link complement $M_{(S,\varphi)} \setminus L$ is independent of a choice of representative $((S,\varphi),L)$ of $\cL$. 
With this in mind,  we say that $\cL$ is {\em Seifert-fibered} (resp. {\em toroidal, hyperbolic}) if  $M_{(S,\varphi)} \setminus L$ is Seifert-fibered (resp. toroidal, hyperbolic).

The next corollary shows that the geometric structure is preserved under taking a branched covering when the FDTC is large compared with the degree of covering.

\Overtwistedcor

\begin{proof}
We prove the theorem under the assumption (a). Under the assumption (b) a parallel argument holds. 

Since 
$|c(\phi,\cL,\partial S)| >  d \geq 1$, 
\cite[Theorem 8.4]{ik2} implies that $M_{(S,\vphi)}\setminus L$ is Seifert-fibered (resp. toroidal, hyperbolic) if and only if the distinguished monodromy $[\varphi_{L}]$ is periodic (resp. reducible, pseudo-Anosov).

Corollary~\ref{cor:NT-type-preserved} says that if 
$[\varphi_{L}]$ is periodic (resp. reducible, pseudo-Anosov) then 
 $\widetilde{\phi}$
is periodic (resp. reducible, pseudo-Anosov).
By Theorem \ref{theorem:covering}, 
$$|c(\widetilde{\phi},\partial \widetilde{S})| = |c(\phi,\cL, \partial S)\slash d(\pi, \partial\widetilde{S})| = |c(\phi,\cL, \partial S) \slash d | >1.$$ 
With \cite[Theorem 8.3]{ik2} we conclude that $M_{(\widetilde{S},\widetilde{\phi})}$ is Seifert-fibered (resp. toroidal, hyperbolic).
\end{proof}

Here is a slightly different application to knot theory. 

\QAcor

\begin{proof}
Assume to the contrary that the closure of $\beta$ is quasi-alternating.
Since being quasi-alternating is invariant under taking the mirror image, that switches the sign of $c(\beta,\partial D)$, we may assume that $c(\beta,\partial D)>2$.

Let $\cL$ be a closed braid in the open book $(D, \id)$ represented by $\beta$. 
Since the monodromy of the open book is $\id$, the distinguished monodromy of $\cL$ is equal to $\beta$. Thus their FDTCs are equal: $c(\id,\cL,\partial D)=c(\beta,\partial D)$.

Take the double branched covering $(\wS, \wphi)\to(D, \id)$ along $\cL$. For each page, it gives a double branched cover $\pi:\wS \rightarrow D$ branched at the puncture points. 
Since the number of the punctures is odd, the surface $\wS$ has connected boundary, i.e., the degree $d(\pi, \partial \wS)=2$.  
By Theorem \ref{theorem:covering}, $c(\widetilde{\phi},\partial \widetilde{S})=\frac{1}{2}c(\beta,\partial D)> 1$. 

Since $\beta$ is pseudo-Anosov, $\widetilde{\phi}$ is also pseudo-Anosov by Corollary \ref{cor:NT-type-preserved}. 
By \cite[Theorem 4.3]{HKM2} the 3-manifold $\widetilde{M}=M_{(\wS,\widetilde{\phi})}$ admits a taut foliation. This contradicts the assumption because the double branched covering of quasi-alternating link is an L-space \cite[Lemma 3.1, Proposition 3.2]{OS2} so it does not admit a taut foliation \cite[Theorem 1.4]{OS}.
\end{proof}

\subsection{Non-right-veering braids and virtually loose transverse links}
\label{section:right-veering-vs-loose}

In this section we study applications of Theorem \ref{theorem:covering} to contact topology.  

We first observe that for a pseudo-Anosov braid with large FDTC its branched covering is universally tight.

\begin{theorem}
Let $\pi:(\wS,\wP) \rightarrow (S,P)$ be a fully ramified branched covering branched at $P$ and $(\wS,\widetilde{\phi})$ be a branched covering of $(S,\phi)$ along a pseudo-Anosov braid $\cL$. For each boundary component $C_i$ of $S$ suppose that  $c(\phi,\cL,C_i)=k_i\slash p_i$ where $p_i$ denotes the number of prongs at $C_i$.
If $k_i\geq 2$ for every boundary component $C_i$, then $(\wS,\widetilde{\phi})$ supports a universally tight contact structure.
\end{theorem}

Although hyperbolicity of $\wS$ is not explicitly assumed in the above statement, the argument in the proof of Theorem~\ref{theorem:covering} shows that the fully ramified condition on $\pi$ and pseudo-Anosov condition on $\cL$ yield $\chi(\wS)<0$. 

\begin{proof}
For each boundary component $\widetilde C_{i_j}$ of the preimage $\pi^{-1}(C_i)$ let $d_{i_j}:=d(\pi, \widetilde C_{i_j})$.   
By Theorem \ref{theorem:covering} and Corollary~\ref{cor:NT-type-preserved}, $\wphi$ is pseudo-Anosov, $c(\widetilde{\phi},\widetilde C_{i_j}) = k_i \slash  p_i d_{i_j}$,  
and the number of the prongs at $\widetilde C_{i_j}$ is exactly $p_i d_{i_j}$. 

Here we recall Colin and Honda's \cite[Corollary 2.7]{CH} on connected boundary case:  
If $c(\phi)=k\slash p$ and $\phi$ is pseudo-Anosov such that $k\geq 2$, then $\xi_{(S,\phi)}$ is universally tight.  
Baldwin and Etnyre extended this result to multi-boundary-component case in  \cite[Theorem 4.5]{BE}.

Since $k_i \geq 2$ for every $i$, applying the Colin-Honda and Baldwin-Etnyre we  conclude that $(\wS,\widetilde{\phi})$ supports a universally tight contact structure. 
\end{proof}

Recall that a transverse link in an overtwisted contact 3-manifold is \emph{loose} if its complement is overtwisted and is \emph{non-loose} otherwise.

\begin{theorem}
\label{theorem:virtually-loose}
Let $\pi:(\wS,\wP) \rightarrow (S,P)$ be a fully ramified covering and $(\wS,\widetilde{\phi})$ be a branched covering of $(S,\phi)$ along a closed braid $\cL$. 
Suppose that $\cL$ is non-right-veering with respect to some  boundary component $C$ of $S$. 
Then $(\widetilde{S},\widetilde{\phi})$ supports an overtwisted contact structure and  the lift $\widetilde{\cL}$ represents a loose transverse link. 
\end{theorem}

Note that we do not required $\chi(\widetilde{S})<0$ in the theorem.
\begin{proof}
Let $((S,\varphi),L)$ be a representative of the closed braid $\cL$. Suppose that $\pi_L: (\wS, \wvphi)\to(S, \vphi)$ represents the given branched covering $(\wS,\wphi) \rightarrow (S,\phi)$. 
Thus $P=p(S_0\cap L)$. 
Let $\widetilde{C}$ be a connected component of $\pi^{-1}(C)$.

Since we assume $\cL$ is non-right-veering at $C$, by Corollary~\ref{cor:covering-right-veering}, both $[\wvphi]\in  \MCG(\wS, \widetilde P)$ and $\wphi =\f[\wvphi] = [\widetilde{\varphi}] \in \MCG(\wS)$ are non-right-veering at $\widetilde C$, where $\widetilde{\varphi}: \wS\to\wS$ is the diffeomorphism obtained from $\wvphi:(\wS, \wP)\to(\wS, \wP)$ by forgetting the punctures.

Hence $(\widetilde{S},\wphi)$ supports an overtwisted contact structure by \cite[Theorem 1.1]{HKM} and there exists a properly embedded arc $\alpha\subset \wS$ such that $\widetilde{\varphi}(\alpha) \prec_{\sf right}^{\widetilde{S}} \alpha$. 
Applying \cite[Lemma 5.2]{HKM} to $\widetilde{\varphi} (\alpha) \prec_{\sf right}^{\widetilde{S}} \alpha$, we can find oriented arcs $\widetilde{\varphi}(\alpha) = \alpha_1,\ldots,\alpha_n=\alpha$
 in $\wS$ starting at the same point on $\widetilde C$ and satisfying: \begin{itemize}
\item[(i)] $\Int(\alpha_i)\cap \Int(\alpha_{i+1})=\emptyset$ for every $i=1, \dots, n-1,$ and 
\item[(ii)] $ \widetilde{\varphi}(\alpha) = \alpha_1 \prec_{\sf right}^{\widetilde{S}} \alpha_2 \prec_{\sf right}^{\widetilde{S}} \cdots \prec_{\sf right}^{\widetilde{S}} \alpha_n=\alpha.$
\end{itemize}
Without loss of generality we may assume that the arcs miss punctures so that they are arcs in $\widetilde{S}\setminus \widetilde{P}$ and $\widetilde{\varphi}(\alpha) = \wvphi(\alpha)$.
Since $\alpha_i$ and $\alpha_{i+1}$ are in tight position in $\widetilde{S}$, so they are in $\widetilde{S}\setminus \widetilde{P}$ and the arcs satisfy (i) and 
\begin{itemize}
\item[(ii)$'$] $\wvphi(\alpha) = \alpha_1 \prec_{\sf right}^{\widetilde{S}\setminus \widetilde{P}} \alpha_2 \prec_{\sf right}^{\widetilde{S}\setminus \widetilde{P}} \cdots \prec_{\sf right}^{\widetilde{S}\setminus \widetilde{P}} \alpha_n=\alpha.$
\end{itemize}

The existence of such a sequence of arcs in $\wS \setminus \wP$ shows that the distinguished monodromy $[\wvphi] \in \MCG(\widetilde{S},\widetilde{P})$ of $\widetilde{L}$ is not quasi-right-veering. See \cite[Section 3]{ik-qveer} for the definition and basic properties of quasi-right-veering. 
\cite[Theorem 4.1]{ik-qveer} states the following: \\

\noindent
{\em A transverse link $\mathcal{T}$ in a contact 3-manifold $(M,\xi)$ is non-loose if and only if every braid representative of $\mathcal{T}$ with respect to every open book decomposition  of $(M,\xi)$ is quasi-right-veering.} \\

\noindent Therefore, we conclude that $\widetilde\cL$ is loose. 
\end{proof}

\begin{definition}
\label{def:univloose}
 We say that a non-loose transverse link $\mathcal{T}$ is \emph{virtually loose} if its complement is virtually overtwisted, that is, the complement admits a (non-branched) finite covering which is overtwisted. Otherwise, we say that $\mathcal{T}$ is \emph{universally non-loose}.
\end{definition}

It is well-known that non-right-veering open books support overtwisted contact structures. However, non-right-veering braids do not represent loose transverse links in general. 
We will show that Theorem \ref{theorem:virtually-loose}, together with the following Lemma~\ref{lemma:existance fully ramified}, the existence of fully ramified branched coverings, relates right-veering closed braids and universally non-loose transverse links.

Let $\mathcal{T}=\mathcal{T}_1\cup \cdots \cup 
\mathcal{T}_{\ell}$ be an $\ell$-component 
transverse link in a contact 3-manifold $(M,\xi)$.
Let $(S,\varphi)$ be an abstract open book supporting $(M,\xi)$ and $L$ be a closed $n$-braid in $(S,\varphi)$ that represents $\mathcal{T}$. 

\begin{lemma}\label{lemma:existance fully ramified}
If $M$ contains no $S^{1} \times S^2$-summands, there exists a fully ramified branched covering $\pi:(\widetilde{S},\widetilde{P}) \rightarrow (S,P)$ and an open book branched covering $\pi_L: (\widetilde{S},\widetilde{\varphi_L}) \rightarrow (S, \varphi)$ along the closed braid $L$.
\end{lemma}

\begin{proof}
Let $\mu_i \in \pi_1(M \setminus \mathcal{T})$ be represented by a meridian of $\mathcal{T}_i$. 
We claim that $\mu_i$ is 
non-trivial. 
If $\mu_i$ were trivial, then it would bound a disk $D$ in $M \setminus \mathcal T$, and the union of a meridional disk of $\mathcal{T}_i$ and $D$ would give rise to a non-separating $2$-sphere in $M$. This contradicts the assumption that $M$ contains no $S^1 \times S^2$. 

Since $\pi_1(M \setminus \mathcal{T})$ is residually finite \cite{Hempel}, there exists a homomorphism
$$f_i: \pi_1(M \setminus \mathcal{T}) \rightarrow G_i$$ for some finite group $G_i$ such that $f_i(\mu_i) \neq 1$. 
Define $F: \pi_1(M\setminus \mathcal{T}) \rightarrow G_1\times \cdots \times G_\ell $
by $F(x)=(f_1(x), \dots, f_\ell(x))$. Putting $G=\textrm{im}(F)$ we obtain a surjective homomorphism 
$$F: \pi_1(M\setminus \mathcal{T})\to G.$$
The argument in \cite[p.3136]{eo} shows that 
$\pi_1(M \setminus \mathcal{T})$ is a quotient group of $\pi_1(S\setminus P)$; thus, there is a surjective homomorphism $$q: \pi_1(S \setminus P) \rightarrow \pi_1(M \setminus \mathcal{T}). 
$$
With these we obtain a surjective homomorphism
$$ F_S= F\circ q: \pi_1(S\setminus P) \rightarrow G.$$
Let 
$\pi: (\widetilde{S}, \widetilde P) \rightarrow (S, P)$ be the branched covering that corresponds to $\ker(F_S)$.

For $p \in P$ let $c_p\in\pi_1(S \setminus P)$ be represented by a small loop around $p$. Since $q(c_p)$ is represented by a meridian of a component of $\mathcal{T}$, we have $F_S(c_p)\neq 1$ by the definition of $F_S$. 
 Since the branched covering $\pi$ is a normal covering corresponding to $\ker(F_S)$, $F_S(c_p) \neq 1$ implies that $c_p$ does not lift to a loop for any choice of base point.
This means that the map $\pi$ restricted to a neighborhood of each preimage $\tilde p \in \pi^{-1}(p) \subset \tilde P$ is a non-trivial (i.e., ramified) branched covering. That is, $\pi: (\widetilde{S}, \widetilde P) \rightarrow (S, P)$ is a fully ramified branched covering.

Since the presentation of $\pi_1(M \setminus \mathcal{T})$
in \cite{eo} shows 
$(q \circ ({\varphi_L})_*)(x)=q (x)$ for $x \in  \pi_1(S \setminus P)$
we have 
$$ F_S\circ ({\varphi_L})_*(x)=F\circ q \circ ({\varphi_L})_*(x) = F \circ q(x) = F_S(x).$$
Hence $(\varphi_L)_*(\ker(F_S)) \subset \ker(F_S)$. Thus by the lifting criterion \cite{Hatcher} $\varphi_L:(S,P) \rightarrow (S,P)$ lifts to a $\widetilde{\varphi_L}:(\widetilde{S},\widetilde{P}) \rightarrow (\widetilde{S},\widetilde{P})$ so we have a branched covering $\pi_L: (\widetilde{S},\widetilde{\varphi}_L) \rightarrow (S,\varphi)$ along $L$.
\end{proof}

Consequently, we obtain a sufficient condition for a transverse link to be virtually loose: 

\VLcor

This answers negatively the {\em transverse version} of Etnyre and Ng's Questions 63 
 and 65 in \cite{EN}:

\VLappcor

\begin{proof}
Let $\mathcal{T}$ be a transverse link in the standard tight contact $S^{3}$ represented by the $3$-braid $\beta=\sigma_1\sigma_2^{-1}\sigma_1\sigma_2^{-1}$ with respect to the disk open book $(D^{2},id)$. 
The topological type of $\mathcal T$ is the {\em figure eight knot}  and $\mathcal{T}$ realizes the maximal self-linking number of the topological type. 

The braid $\beta$ is pseudo-Anosov and has the FDTC  $c(\beta, \partial D^2)=0$. 
By Honda-Kazez-Mati\'c's right-veering criterion (Theorem~\ref{theorem:HKM-RV-criterion}) $\beta$ is not right-veering.
Hence, $\mathcal T$ is virtually loose by Corollary ~\ref{cor:virtually-loose2}. Indeed, it is easy to see that its double covering is already overtwisted. 
\end{proof}

\begin{remark}\label{rem:Question64}

The transverse link version of \cite[Question 64]{EN} asks \emph{if a transverse link $\mathcal{T}$ has non-maximal self-linking number but is not a stabilization, then is $\mathcal{T}$ universally non-loose ?}

There are several examples of non-destabilizable transverse links with non-maximal self-linking number. 
According to the Legendrian knot atlas \cite{CN, N}, $m(10_{145}),m(10_{161})$ and $12n_{591}$ contains such examples. However, their braid representatives from the grid diagrams in \cite{N} are right-veering so we cannot see whether they are virtually loose or not. 
\end{remark}

By analogy with the relation between tightness and right-veering-ness  it is interesting to ask the converse of Corollary \ref{cor:virtually-loose2}: 
\begin{question}
\label{question:main}
Let $\mathcal{T}$ be a transverse link in a contact 3-manifold $(M,\xi)$. 
Is it true that $\mathcal{T}$ is universally non-loose if and only if every closed braid representative $\cL$ of $\mathcal T$ with respect to every open book decomposition of $(M, \xi)$ is right-veering?
\end{question}

{
\address{Department of Mathematics, Kyoto University, Kyoto 606-8502, JAPAN}
\email{tetitoh@math.kyoto-u.ac.jp}

\address{Department of Mathematics,The University of Iowa, Iowa City, IA 52242, USA}
\email{keiko-kawamuro@uiowa.edu}
}

\end{document}